\documentclass[a4paper, 12pt]{amsart}

\usepackage{amssymb}

\usepackage{ textcomp }
\mathchardef\emptyset="001F

\theoremstyle{plain}

\newtheorem{theorem}{Theorem}[section]
\newtheorem{lemma}[theorem]{Lemma}
\newtheorem{proposition}[theorem]{Proposition}

\newtheorem{remark}[theorem]{Remark}
\newtheorem{definition}[theorem]{Definition}
\newtheorem{example}[theorem]{Example}
\newtheorem{hypothesis}[theorem]{Hypothesis}

\numberwithin{equation}{section}

\newcommand{\e}{\epsilon}
\newcommand{\eps}{\epsilon}

\newcommand{\R}{{\mathbb R}}
\newcommand{\Op}{{\mathcal O}}

\newcommand{\N}{{\mathbb N}}
\newcommand{\Q}{{\mathbb Q}}

\newcommand{\ZZ}{{\mathbb Z}}
\newcommand{\KK}{{\mathbb K}}
\newcommand{\NN}{{\mathbb N}}
\newcommand{\QQ}{{\mathbb Q}}

\newcommand{\DD}{{\mathbb D}}
\newcommand{\MM}{{\mathbb M}}
\newcommand{\RR}{{\mathbb R}}
\newcommand{\CC}{{\mathbb C}}
\newcommand{\LL}{{\mathbb L}}

\newcommand{\T}{{\mathbb T}}

\newcommand{\TT}{{\mathbb T}}

\newcommand{\V}{{\mathcal V}}
\newcommand{\VV}{{\mathcal V}}
\newcommand{\caS}{{\mathcal S}}

\newcommand{\caH}{{\mathcal H}}

\newcommand{\dive}{{\rm div}}

\newcommand{\no}{\noindent}

\renewcommand{\div}{\hbox{{\rm div}}}
\newcommand{\grad}{\hbox{{\rm grad}}}

\newcommand{\bfv}{{\bf v}}
\newcommand{\bfn}{{\bf n}}
\newcommand{\bfe}{{\bf e}}
\newcommand{\bfb}{{\bf b}}

\newcommand{\caT}{{\mathcal T}}

\newcommand{\caO}{{\mathcal O}}

\def\2s{\stackrel{\rm 2s}{\rightharpoonup}}

\def\uu{{\underline u}}
\def\uv{{\underline v}}

\def\utheta{{\underline \theta}}
\def\uepsilon{{\underline \epsilon}}
\def\curl{{\rm curl}\,}

\title[]
{Diffraction of Bloch  Wave Packets for Maxwell's Equations}

\author[Gr\'egoire Allaire]{Gr\'egoire Allaire$^{1}$}
\author[Mariapia Palombaro]{Mariapia Palombaro$^{2}$}
\author[Jeffrey Rauch]{Jeffrey Rauch$^{3}$}

\begin{document}
\baselineskip3.15ex
\vskip .3truecm

\begin{abstract}
\small{
We study, for times of order $1/h$,
solutions of Maxwell's equations 
in an  
$\Op(h^2)$ modulation
of an $h$-periodic medium. The solutions 
are of slowly varying amplitude type built on
Bloch plane waves with wavelength of order $h$.
We construct  accurate approximate solutions
of three scale WKB type.
The leading profile is both transported at the group
velocity and dispersed by a Schr\"odinger equation given
by the quadratic approximation of the Bloch dispersion
relation.
A weak ray average hypothesis
guarantees stability.  Compared to earlier
work on scalar wave equations,  
the generator is no longer elliptic.
Coercivity  holds only on the complement
of an infinite dimensional kernel.
The system
structure  
 requires
many innovations.
}
\end{abstract}

\maketitle
{
\small
\noindent
$^1$ Centre de Math\'ematiques Appliqu\'ees, \'Ecole Polytechnique, 
91128 Palaiseau, France.\\
Email: gregoire.allaire@polytechnique.fr\\
\noindent
$^2$ 
University of L'Aquila, Department of Information Engineering,
Computer Science and \\
Mathematics,
Via Vetoio 1 (Coppito), 67100 L'Aquila, Italy.\\
Email: mariapia.palombaro@univaq.it

\noindent
$^3$ Department of Mathematics, University of Michigan,
Ann Arbor 48109 MI, USA.\\
Email: rauch@umich.edu
}

\maketitle

\section{Introduction.}

This paper studies  the propagation of 
electromagnetic waves 
through perturbed periodic
media with period $h\ll 1$.%
\footnote{A word on units.
The Maxwell  equations in vacuum
have 
permittivities $\epsilon, \mu$ 
and speed of light $c=1/\sqrt{\epsilon\mu}$.
Since the speed of light  is $\gg 1$ in KMS
or CGS units, 
$\eps\mu$ is small
 in those units.   We perform an asymptotic
analysis as $h\to 0$. Denote by 
$\Delta t$ the unit of time.
No matter what the 
units one has
$h\ll \Delta t/ \{\eps, \mu\}$ in this limit 
so there is scale separation
no  matter what are the values of $\eps$
and $\mu$.
Nevertheless it is wise to, and we choose 
to,  work in units
with $c\,\Delta t$ 
comparable to 1.  For example centimeters 
for length and $\Delta t\ll 1$
equal to
the number of seconds
that it takes light in vacuum to traverse one centimeter.
In those units $c\Delta t=1$
and one expects that the constants in 
our
error bounds will not  be very large.
}
 We
treat the 
resonant case where the
length scale of the periodic structure is comparable
to the  wavelength.
The  observation time $t$  satisfies $t\sim 1/h$. 
This is the diffractive time scale where 
standard Maxwell 
equations (without periodic structure)
are approximated by Schr\"odinger's equation
(see \cite{donnat1}, \cite{jmr}).
  Wavelengths that are
short compared to the period are short compared to the scale on which
the coefficients vary.  This is the domain of validity of 
standard geometric optics
(see \cite{donnat1}, \cite{lannes}
 for the diffractive case).
Wavelengths long compared to the period are analysed by
standard homogenisation
\cite{blp}.
  The interest, both mathematical
and scientific, of the resonant  scaling is that the speeds of propagation
and diffractive effects for wave packets
are given by the Bloch  dispersion
relation of the periodic medium and not
by the symbol of  the
 hyperbolic operator or its hyperbolic homogenisation.
    The propagation
    speeds can be radically different from those of 
the
original equations. The new speeds must not violate the finite
speed of the original equations (see 
Theorem \ref{thm.speed} for a proof of this upper bound) but can be much smaller.
This is the basis for strategies to  slow light
(\cite{altug}, \cite{hau}, \cite{bajcsy}, \cite{vlasov}).  That in turn is
one of the proposed  design elements of the all
optical computer. 
Another domain of application is photonic crystal
fibers constructed with periodicity in crossection
(see \cite{russell}, \cite{gersen}, \cite{kuchment2}).
The last three examples are modeled by Maxwell's equations
that are
the subject of the current article.
 
 In our papers \cite{apr}, \cite{apr2}  we study
 scalar wave equations.
 In most cases,  Maxwell's
equations cannot be reduced to scalar equations.  
The most interesting applications require the 
methods of the present paper.
For constant {\it scalar} permittivities 
the  Maxwell equations can be reduced to  the scalar
wave equation.     For constant
$\epsilon,\mu$ with $\epsilon$ a three by three matrix with 
three distinct positive eigenvalues, it has been known
since the time of Hamilton 
(see \cite{hamilton},  \cite{courant}
page 610),
 that the characteristic
polynomial is of the form
$
\tau^2\,Q(\tau,\xi)
$
with $Q$ an irreducible quartic polynomial.  The characteristic
variety is conic with nontrivial singular points.
$Q$ does
not factor as the product of two quadratics in which case
the variety would be the union of two (double) cones with elliptic
cross-section.
For variable, even scalar, permittivities,
the standard derivation of second order equations
by taking time derivatives works but leads to a 
{\it system} of second order equations for $E$ coupled
through lower order terms.

Maxwell's dynamic equations  for unknown $E(t,x),B(t,x)\in \RR^3\times\RR^3$
read
\begin{equation}
\label{eq:Maxwell}
\partial_t
\begin{pmatrix}
\epsilon\, E \cr
\mu\, B
\end{pmatrix}
\ +\ 
\begin{pmatrix}
-\,\curl B \cr
+\, \curl E
\end{pmatrix}
= 0 
\end{equation}
whose infinitesimal generator  is not elliptic. 
Indeed if $\epsilon,\mu$ depends only on $x$, 
then $(\nabla_x\phi\,,\,\nabla_x\psi)$
is a  stationary solution for any 
$\phi(x),\psi(x)\in C^\infty_0(\RR^3)$. 
For any ball there is an infinite dimensional set
of such solutions 
supported in the ball.
  This is one of the main differences 
of the current article with the earlier ones on 
scalar wave equations.   
The failure of ellipticity is compensated by the fact
that the set of physically relevant
 solutions
have additional strong control on their divergence.
Those solutions satisfy semiclassical ellipticity estimates
(see Theorem \ref{thm:coercivity1}).  A second principal
difference with the earlier work is that for systems
it is not uncommon to have Bloch eigenvalues of multiplicity
greater than one.  In that case the Schr\"odinger equations
of diffractive geometric optics are systems and we
treat that possibility.
A third difference with the scalar case is a simplification.
For the scalar case,  estimates for gradients 
of the error in the approximate solution were
straight forward but there was a strikingly difficult
argument to estimate the undifferentiated error.
In the case of systems, the natural energy estimate
is an  $L^2$ estimate.   Derivative estimates are proved
from $L^2$ estimates for the time derivatives 
and the divergence.  The
remaining derivatives
are estimated by an ellipticity argument.    
The difficult argument 
 from the scalar case is not required.

There are two problems  closely related to the ones
we analyse.
The first is the 
treatment of first order elliptic systems.  The second is the 
treatment of the Maxwell system on the time scales of 
geometric optics.   The second problem is solved
{\it en  passant} in \S \ref{sec:gopBloch}.  We have chosen to skip the first
and  jump directly to the 
Maxwell's equations.   The methods that suffice for Maxwell
yield the 
elliptic case 
directly.

The semiclassical estimates of the present
article permit a strengthening of 
the 
earlier paper 
\cite{apr2},
where 
derivatives of order $\le 1$ were estimated.
The new method yields estimates for 
 derivatives
of all orders.    

Our three articles use corrector terms
in asymptotic expansions.
Article  \cite{apr} 
correctors
were constructed by an {\it ad hoc}
method and used in  test functions 
to study 
 weak convergence.   In \cite{apr2}
we introduced a
a general strategy
yielding  accurate expansions.   

The $h$-dependent Maxwell equations are written in the form
$P^h(E^h,B^h)=0$ with 
\begin{equation}
\label{eq:maxwelldiffractive}
P^h(t,x,\partial_t,\partial_x) (E^h,B^h) \, := 
\partial_t
\begin{pmatrix}
\epsilon^h\, E^h \cr
\mu^h\, B^h
\end{pmatrix}
\ +\ 
\begin{pmatrix}
-\,\curl B^h \cr
+\, \curl E^h
\end{pmatrix}
\ +\ M^h
\begin{pmatrix}
E^h \cr
B^h
\end{pmatrix} \,.
\end{equation}
For the diffractive scaling the perturbations satisfy the following
hypothesis  where $\epsilon_0(x/h)$ and  $\mu_0(x/h)$ 
are the permittivities of the  unperturbed
periodic structure at scale $h$.
The  $6\times 6$ matrix valued function $M^h$ serves
for example to model dissipative effects such as Ohm's
law (see \S  \ref{sec:gopBloch}).

\vskip.2cm

\no
{\bf Notation.} {\sl For two vectors $(e,b)\in\CC^3\times\CC^3$, their Hermitian inner
product is denoted by $\langle e,b\rangle$ while $(e,b)$ denotes the
ordered pair
 in $\CC^3\times\CC^3$. The cross product in $\CC^3$ is denoted by $\wedge$. 
The set of linear maps (homeomorphisms) on a vector space $K$ is denoted by 
${\rm Hom}(K)$. For any $\alpha=(\alpha_0,\alpha_1,\alpha_2,\alpha_3)\in\NN^4$ 
the notation $\partial^\alpha_{t,x}\phi(t,x)$ means 
$\partial^{\alpha_0}_t\left(\prod_{i=1}^3 \partial^{\alpha_i}_{x_i}\right)\phi(t,x)$. 
}

\vskip.3cm

The Maxwell equations \eqref{eq:maxwelldiffractive} 
are a 
symmetric first-order hyperbolic system for $u^h=(E^h,B^h)$
\begin{equation}
\label{eq:maxwelldiffractive2}
P^h(t,x,\partial_t,\partial_x) u^h \, = \,
\partial_t (A_0^h u^h )+ \sum_{j=1}^3 A_j \partial_{x_j} u^h 
\ +\ M^hu^h\,,
\end{equation}
with symmetric matrix coefficients $A_j$ for $j\ge 1$  defined in
\eqref{eq:AintermsofJ}, \eqref{eq:defJ} and 
\begin{equation*} 
 A_0^h(t,x) 
 \ :=\
  \left(
\begin{array}{ll}
\epsilon^h(t,x) & 0\\
0 & \mu^h(t,x)
\end{array}
\right)
\,.
\end{equation*}

\begin{hypothesis}  
\label{hyp:diffractive}
{\bf Diffractive time scale hypothesis.}
The coefficients in \eqref{eq:maxwelldiffractive}
are given by
$$
\begin{array}{lcl}
\epsilon^h(t,x) &=& \epsilon_0(x/h) +h^2\epsilon_1(t,x,x/h) \,, \\
\mu^h(t,x) &=& \mu_0(x/h) +h^2\mu_1(t,x,x/h) \,, \\
M^h(t,x) &=& h\,M(t,x,x/h)\,.
\end{array}
$$
The matrix valued functions $\epsilon_0(y)\,,\,
 \mu_0(y)\in C^\infty(\TT^3)$
 are
 symmetric  
and positive definite and 
for all $\alpha$, 
$\partial_{t,x,y}^\alpha \{\epsilon_1,\mu_1, M\}(t,x,y)\in L^\infty(\RR^{1+3}\times\TT^3)$ with $\TT^3$ denoting the three-dimensional torus $(\RR/2\pi \ZZ )^3$.
Define
$$
A_0^0(y)
\ :=\
\begin{pmatrix}
 \eps_0(y) & 0
\cr
0 & \mu_0(y)
\end{pmatrix},
\qquad
A_0^1(t,x,y) 
\ :=\
\begin{pmatrix}
\eps_1(t,x,y) & 0
\cr
0 & \mu_1(t,x,y)
\end{pmatrix}
\,.
$$
\end{hypothesis}

\vskip.2cm

\begin{definition}
\label{def:L2periodic}
The space of $L^2_{loc}(\R^3)$ periodic functions of
period $2\pi$ is denoted $L^2(\T^3)$. The functions 
$(E,B)$ of $L^2(\T^3)$ with values in $\R^3\times\R^3$ 
are normed by 
$$
\int_{[0,2\pi[^3} \Big( |E|^2 + |B|^2\Big)  dx \,.
$$
When normed by the equivalent expression natural
in the context of Maxwell's equations 
$$
\int_{[0,2\pi[^3} \Big(
\langle E\,,\, \epsilon_0 \, E\rangle 
\  +\ 
\langle B\,,\, \mu_0\,B\rangle \Big)
\ dx
$$
it is denoted $L^2_{\epsilon_0,\mu_0}(\T^3)$.
\end{definition}

\begin{definition}
\label{def:thetaper}
For $\theta\in [0,1[^3$ a function $g(x)$ on $\R^3$ is {\bf $\theta$-periodic}
when 
$x\to e^{- i\theta.x}\,g(x)$ is periodic in $x$ with period $2\pi$.
 The  parameter $\theta$ is called 
the {\bf Bloch frequency}.  
The set of $L^2_{loc}(\R^3)$ $\theta$-periodic
functions is denoted $L^2(\T^3_\theta)$. With
alternate norm as in 
Definition \ref{def:L2periodic}
 it is denoted
$L^2_{\epsilon_0,\mu_0}(\T^3_\theta)$.
\end{definition}

Our solutions are amplitude modulated Bloch plane waves.  The
plane waves are
$\theta$-periodic
solutions of the unperturbed periodic Maxwell equations 
of the form $u=e^{\lambda t}(E(x),B(x))$.
Equivalently $E,B$  are solutions of the  
spectral problem
\begin{equation}
\label{eq:evalue}
\lambda\,\begin{pmatrix}
\epsilon_0(x) & 0
\cr
0 & \mu_0(x)
\end{pmatrix}
\begin{pmatrix}
E(x)\cr B(x)
\end{pmatrix} 
\ =\ 
\begin{pmatrix}
 0 & \curl 
\cr
-\curl & 0
\end{pmatrix}
\begin{pmatrix}
E\cr B
\end{pmatrix}\,,
\qquad
\{ E\,,\, B\} \quad \theta-{\rm periodic}\,. 
\end{equation}

We recall in \S \ref{sec:Bloch-theory}
that the spectrum at fixed $\theta$ consists of 
$\{0\}$ with infinite multiplicity and a discrete set of
purely imaginary
eigenvalues $\lambda=i\omega(\theta)$ 
of finite multiplicity.  We label the nonzero
eigenvalues
according to their distance
from the origin and  repeat them according to their 
multiplicity
\begin{equation}
\label{eq:numbering}
\cdots\ \ \le \ \omega_{-2}(\theta)\ \le \ 
\omega_{-1}(\theta)\ < \ \omega_0(\theta)=0\ <\
\omega_{1}(\theta)
\ \le \
\omega_{2}(\theta)
\ \le \ 
\cdots\,.
\end{equation}
We work near an eigenvalue 
of constant multiplicity.

\begin{hypothesis}
\label{hyp:constantmultiplicity}
({\bf Constant\ multiplicity\ hypothesis.})  
Fix 
$\utheta\ne 0$, $n\in\ZZ^*$, and
denote by $\kappa$ the multiplicity,
$$
\underline\lambda
\ =\
i\omega_n(\utheta)
\ \ne\
 0\,,
\qquad
\omega_{n-1}(\utheta)
\ <\
\omega_n(\utheta) 
\ =\  \dots\  =\
\omega_{n+\kappa-1}
(\utheta) 
\ <\
\omega_{n+\kappa}
(\utheta)\,.
$$
Assume that  there are $\delta_1>0$, $\delta_2>0$
and real analytic $\omega(\theta)$
defined on $\{|\theta-\utheta|<\delta_1\}$ 
so that for each $|\theta-\utheta|<\delta_1$ the only point of the spectrum
in $\{\lambda\,:\,|\lambda-i\omega_n(\utheta)|<\delta_2\}$  is 
$i\omega(\theta)$ with multiplicity $\kappa$.
\end{hypothesis}

\noindent
The hypothesis is automatically satisfied when $\kappa=1$.
An example with $\kappa >1$ is $\eps$ and $\mu$  constant and scalar where every
eigenvalue is of constant multiplicity two (see Example \ref{important-example} below).

\begin{definition}
\label{def:groupvelocity}
When the constant multiplicity hypothesis
\ref{hyp:constantmultiplicity} is satisfied
the {\bf group velocity} is defined by
$$
\V\ :=\ 
-\nabla_\theta \omega(\utheta)\,.
$$
\end{definition}

Define
\begin{equation}
\label{defL}
\LL(\omega, \theta, y,\partial_y)
\ :=\ i \omega
\begin{pmatrix}
\epsilon_0(y)&0 \cr
0& \mu_0(y)
\end{pmatrix}
- 
\left(
\begin{array}{cc}
0 & (i\theta + \partial_y )\wedge \\
-(i\theta + \partial_y )\wedge & 0
\end{array}
\right)
\,.
\end{equation}
$\LL$ 
with domain equal to the periodic
functions in 
$H^\infty(\TT^3_y)
:=\cap_{s\geq0} H^s(\TT^3)$
 is formally 
antiselfadjoint 
on $L^2(\TT^3\,;\, dy)$.
The method of proof of Proposition \ref{prop:antisa}
shows that  the closure has domain equal to the
$v\in L^2(\TT^3)$ so that $\LL v\in L^2(\TT^3)$
and that the closure is antiselfadjoint. 

\begin{definition} 
\label{def:L-P}
Denote by $\Pi$ the projection operator onto
$\KK := \ker \LL(\omega(\utheta), \utheta,y,\partial_y)$
along the image of $\LL$.  $\Pi$ is
orthogonal with respect to the scalar
product of
$L^2(\TT^3\,; \,dy)$ and {\bf not}
with respect to the scalar product of
$L^2_{\epsilon_0,\mu_0}(\TT^3)$.
\end{definition}

Our wave packets have group velocity $\V$ and travel
for long times.   They see the coefficients on 
group lines for long times.
   The averages of the coefficients along such
long rays are important.  A particular combination 
enters in the asymptotic description.

\begin{definition} 
\label{def:gamma}
For each $t,x$ define the linear map
$\gamma(t,x)\in {\rm Hom}(\KK)$ by
\begin{equation}
\label{eq:defgamma}
\gamma(t,x)
\ :=\
\big(
\Pi A_0^0(t,x)\Pi\big)^{-1}\ 
 \Pi\, ( i \omega A_0^1(t,x)+ M(t,x))\, \Pi
 \,.
 \end{equation}
\end{definition}

We assume that the ray averages
\begin{equation}
\label{eq:meangamma}
\lim_{T\to +\infty}\
\frac{1}{T}
\int_0^T\, \gamma(t,x+\V t)\,dt
\ :=\
\widetilde\gamma(x)\,,\qquad
{\rm
exist \ uniformly\ in\ }\ x\in\RR^3.
\end{equation}
We make a fairly weak assumption asserting that 
this limit is attained at an algebraic rate.

\begin{definition}
\label{def:rayaverage}
The function $\gamma$ satisfies the 
{\bf ray average hypothesis} when 
 \eqref{eq:meangamma} holds and
there is a $0\le \beta<1$ so that
for all $\alpha\in \NN\times\NN^3$ the solution $g_\alpha(t,x)$ of
\begin{equation*}
\Big(
\partial_t \ +\
\V.\partial_x\Big)g_\alpha
\ =\
\partial_{t,x}^\alpha\big(
\gamma(t,x) -\widetilde\gamma(x-\V t)
\big)\,,
\qquad
g_\alpha(0,x)\ =\ 0
\end{equation*}
satisfies 
$\langle t\rangle^{-\beta}g_\alpha\in L^\infty([0,\infty[\times\RR^3)$ 
where $\langle t\rangle:=(1+t^2)^{1/2}$.
\end{definition}
\noindent
This hypothesis, introduced in \cite{apr2}, is discussed in 
\S \ref{sec:rayaverages}.
Our main theorem
gives an approximate solution and
an error estimate.
In the  theorem, $\caT$ is a new
variable, a slow time.
In the approximate solution
it is replaced by $ht$.
In addition there is a $\KK$
valued function, $\widetilde w_0(\caT,x)$.
Abusing notation, the value at $(\caT,x)$
is a function of $y$ denoted
$\widetilde w_0(\caT,x,y)$.

\begin{theorem}\label{thm:main-diffr}
Assume that $\gamma$ satisfies the ray average hypothesis
with 
parameter $0\le \beta<1$. 
For $f\in \caS(\R^3;\KK)$ define
$w_0:=\widetilde w_0(\caT,x-\V t)$,
where
$\widetilde w_0\in C^\infty(\R \,;\, \caS(\R^3;\KK))$ 
is the unique
solution of the initial value problem for 
Schr\"odinger's equation
\begin{equation}
\label{eq:transp-schrod}
\Big(\partial_\caT   + 
\frac{1}{2}i \  
\partial^2_\theta\omega(\partial_x,\partial_x)  
+\widetilde \gamma(x) 
\Big)\widetilde w_0 
\ = \
0,
\qquad
\widetilde w_0(0,x)=f(x).
\end{equation}
Define a family of  approximate solutions
$$
\uv^h(t,x)\ :=\ e^{i(\omega(\utheta) t+\utheta .x)/{h}}\  w_0(ht, t,x,x/h)\,
$$
and let $\uu^h$ denote the exact solution of $P^h\uu^h=0$
with $\uu^h|_{t=0}=\uv^h|_{t=0}$.
Then
\begin{equation*}
\forall\, T>0,\,\alpha\in\NN^4,\,  \exists \, C(\alpha,T),\,
\forall \, h\in ]0,1[,
\quad
\sup_{t\in [0,T/h]}\ 
\big\|
(x\,,\,h\,\partial_{t,x}  )^\alpha
(\uu^h - \uv^h)
\big\|_{L^2(\R^3)} \ \leq\  C\,h^{1-\beta}  \,,
\end{equation*}
with  $0\leq\beta<1$ from the ray 
average hypothesis of Definition \ref{def:rayaverage}.
\end{theorem}

\begin{remark}
{\bf i.}  The operator
 $\Pi A_0^0\Pi$ is positive definite on $\KK$.
When $M=0$,
mulitplying  
the Schr\"odinger equation (1.9) 
by $\Pi A_0^0\Pi$ shows that  
$\int \langle \Pi A_0^0 \Pi \widetilde w_0\,,\,\widetilde w_0 \rangle\,dx$
is conserved.
{\bf ii.}   The principal part of equation 
\eqref{eq:transp-schrod}
is scalar. Coupling occurs through $\widetilde \gamma$
(see Remark \ref{rmk:scalar}).
\end{remark}

\vskip.2cm

\no
{\bf Acknowledgements.}  The research of G. Allaire was  partially supported by 
the DEFI project at INRIA Saclay Ile de France.
G. Allaire and J. Rauch were partially supported
  by the Chair Mathematical modelling and numerical simulation, F-EADS - Ecole Polytechnique - INRIA.
J. Rauch was partially
supported by the U.S. National Science Foundation
under grant  NSF-DMS-0104096 and by the GNAMPA.
  M. Palombaro and
J. Rauch  thank the
CMAP at the \'Ecole Polytechnique and the Centro di Ricerca Matematica 
``Ennio De Giorgi" 
 for their  hospitality.

\section{$L^2(\RR^3)$ estimates}
\label{sec:energy}

Suppose given 
real symmetric
permittivities 
$$
\epsilon^h(t,x)
\,,\,
\mu^h(t,x)
\  \in \ 
C^\infty\big(\RR^{1+3}\times]0,1[_h\,;\, {\rm Hom}(\R^3) \big)
$$
satisfying the positivity constraint
\begin{equation}
\label{eq:posdef}
\forall\ t,x,h\qquad
0
\ <\ 
cI
\ \le\
\epsilon^h(t,x)\,,\,\mu^h(t,x)
\ \le \
CI\,.
\end{equation}
Consider the dynamic Maxwell equations
in a medium that varies on scale $0<h<<1$ 
\begin{equation}
\label{eq:Maxw3}
\partial_t (\epsilon^h(t,x)\, E^h) \ =\ \curl B^h,
\qquad
\partial_t (\mu^h(t,x) B^h) \ =\ -\,\curl E^h\,.
\end{equation}
This is a symmetric hyperbolic system.    
The energy identity for solutions is
\begin{equation}
\label{eq:energy}
\partial_t\int_{\RR^3} 
 \langle E^h,\epsilon^h E^h\rangle +  \langle B^h,\mu^h B^h\rangle 
 \ dx
\ =\ 
- \int_{\RR^3} 
 \langle E^h,\partial_t \epsilon^h E^h\rangle\ +\
\langle B^h,\partial_t \mu^h B^h\rangle\ dx\,.
\end{equation}
Large time derivatives of $\epsilon^h, \mu^h$
can lead to rapid growth of energy.
On the other hand if $\partial_t\epsilon^h,\partial_t\mu^h$ are bounded
(resp.   $\Op(h)$)
one has uniform estimates for the $L^2$ norm
for $t=\caO(1)$  (resp. $t=\caO(1/h)$).  Corresponding 
estimates for derivatives is 
subtle  because the 
coefficients are rapidly varying in space.   Our 
strategy is
to derive estimates for time derivatives 
and for $\div E^h, \div B^h$.
Then estimate
spatial derivatives using an elliptic estimate from
 the next 
section.

\section{Coercivity}
\label{sec:coercivity}

The generator of the dynamic equations \eqref{eq:maxwelldiffractive}
is not elliptic.  However, the special structure of Maxwell's
equations yields supplementary bounds on 
$\dive E$ and $\dive B$.   
The over determined system consisting of the generator
together with divergence is elliptic.
For problems with coefficients oscillating in 
$x$ on scale $h$, estimates in 
semiclassical Sobolev spaces are natural.  

\begin{definition}
\label{def:hsobolev}
The semiclassical Sobolev norm $H^m(\RR^3)$
denoted $\| \cdot\|_{H^m_h(\R^3)}$
is defined by
$$
\bigg(
\sum_{ |\alpha|\le m }
\int \big|(h\partial_x)^\alpha u\big|^2\ dx
\bigg)^{1/2}\,.
$$
For a function $u(t,x)$, integer $m\ge 0$ and
$h\in ]0,1[$ define the semiclassical norm
with 
derivatives in space and time,
\begin{equation}
\label{eq:normdef}
\big\|
u(t)\big\|_{m,h}^2
\ :=\
\sum_{|\alpha|\le m}
\ \big\| (h\,\partial_{t,x}  )^\alpha
  u(t)
\big\|_{L^2(\RR^3)}^2\,.
\end{equation}
Denote by $C^k_h(\RR^{1+3})$ the set of families
$\{w^h\,:\,0<h<1\} \subset C^k(\RR^{1+3})$  so that
\begin{equation}
\label{eq:cinfnorm}
\sup_{0<h<1}\ 
\sum_{|\alpha|\le k}
\ 
\big\| (h\,\partial_{t,x}  )^\alpha
  w^h
\big\|_{L^{\infty}(\RR^{1+3})}
\ <\ \infty.
\end{equation}
The same notation is used  for functions valued
in any finite dimensional real or complex vector space.
The left hand side of \eqref{eq:cinfnorm} serves as norm
making $C^k_h$ a Banach space.
 The Fr\'echet space $C^\infty_h$ is defined as $\cap_kC^k_h$.
\end{definition}

\begin{theorem}
\label{thm:coercivity1}  For $\epsilon^h(t,x), \mu^h(t,x)\in 
C^\infty_h(\RR^{1+3})$ and $m\in \NN$
there is a constant 
$C(m)$ so that for all $E,B\in H^m(\R^3)$, $t\in \R$, and $h>0$
$$
\big\| E(t)\big\|_{H^m_h(\R^{3})}
\le 
C(m)\,
\Big(
\big\| h\,\curl E(t)\big\|_{H^{m-1}_h(\R^3)}
+
\big\| h\,\div (\epsilon^h(t,x) E(t))\big\|_{H^{m-1}_h(\R^3)}
 +
\big\| E(t)\big\|_{H^{m-1}_h(\R^3)}
\Big)\,,
$$
$$
\big\| B(t)\big\|_{H^{m}_h(\R^3)}
\le 
C(m)\,
\Big(
\big\| h\,\curl B(t)\big\|_{H^{m-1}_h(\R^3)}
+
\big\| h\,\div (\mu^h(t,x) B(t))\big\|_{H^{m-1}_h(\R^3)}
 +
\big\| B(t)\big\|_{H^{m-1}_h(\R^3)}\Big)\,.
$$
\end{theorem}

\begin{proof}  The inequalities for $E$ and $B$ are identical so it suffices
to consider $E$.  The variable $t$ is simply a parameter so it suffices to
consider $t$ fixed.  The key estimate is the following.  

\begin{lemma}
\label{lem:coercivity}
If $\uepsilon(x)$
is a symmetric matrix valued function 
so that $\uepsilon \ge cI>0$ for all $x$,
and $\partial_x^\alpha\uepsilon\in L^\infty(\RR^3)$
for all $\alpha\in \NN^3$,
then for each $m\in \NN$ there is a constant $C(m)$
so that for all $E\in H^m(\R^3)$
\begin{equation}
\label{eq:elliptic}
\big\| E\big\|_{H^m(\R^{3})}
\ \le \
C(m)\,
\Big(
\big\| \curl E\big\|_{H^{m-1}(\R^{3})}
\ +\ 
\big\| \div (\underline\epsilon(x) E)\big\|_{H^{m-1}(\R^{3})}
\ +\
\big\| E\big\|_{H^{m-1}(\R^{3})}
\Big)\,.
\end{equation}
\end{lemma}

\begin{proof}   The integrand in 
$$
\int_{\RR^3}
|\div(\uepsilon(x)\, E)|^2
\ +\ 
|\curl E|^2\ dx
$$
is a quadratic form in $E$ and its first derivatives.   The terms
quadratic in the derivatives of $E$ have the form
$$
\sum_{1\le i,j\le 3}\ \big \langle a_{i,j}(x)\, \partial_i E\,,\, \partial_j E\big\rangle\
$$
with uniquely determined real  matrix valued 
functions $a_{i,j}(x)$ with $a_{j,i}$ equal to the transpose of 
$a_{i,j}$.
  Introduce the the symbol
$
\sum_{i,j} a_{i,j}(x) \xi_i\xi_j 
  $.   
  The definition implies that for any  $x$, $\bfe$ and $\xi$
  in $\RR^3$
  \begin{equation}
  \label{eq:quadform}
 \sum_{i,j} \big\langle  a_{i,j}(x)\xi_i\xi_j \bfe\,,\, \bfe \big\rangle
   \ =\ 
| 
\langle\xi\,,\,\uepsilon(x)\bfe\rangle
|^2
\ +\ 
|
\xi\wedge \bfe|^2\,.
\end{equation}
Choose $0<c\le 1$ so that for all $x$, $\uepsilon(x)\ge cI$.
The lemma is a consequence of G\"arding's inequality once
we prove that for all real
$x,\bfe,\xi$
\begin{equation}
\label{eq:pregarding}
| 
\langle\xi\,,\,\uepsilon(x)\bfe\rangle
|^2
\ +\ 
|
\xi\wedge \bfe|^2
\ \ge\
c\, |\xi|^2\, |\bfe|^2/2\,.
\end{equation}
By homogeneity it suffices to consider
$|\xi |=1$.

Decompose $\bfe=\bfe_\parallel + \bfe_\perp$
into parts parallel and perpendicular to $\xi$. 
The definition of $c$ yields
\begin{equation}
\label{eq:firsthalf}
|\bfe_\perp|^2 \le  |\bfe |^2/2
\quad
\Longrightarrow
\quad
|
\langle \xi\,,\, \uepsilon(x) \bfe_\parallel
\rangle
|^2
\ \ge \  
c\, | \bfe_\parallel |^2
\ =\ 
c\big(|\bfe|^2-|\bfe_\perp|^2
\big)
 \ \ge\ c\,
 |\bfe|^2/2
\,.
\end{equation}
   On the other hand since $c\le 1$
\begin{equation}
\label{eq:secondhalf}
|\bfe_\perp|^2\ge|\bfe|^2/2
\quad
\Longrightarrow
\quad
|\xi\wedge \bfe|^2 \ =\ 
|\bfe_\perp|^2
\ \ge\ 
|\bfe|^2/2
\ \ge\ 
c\,|\bfe|^2/2\,.
\end{equation}
Estimates \eqref{eq:firsthalf}
and \eqref{eq:secondhalf}
imply
\eqref{eq:pregarding}
completing the proof of the lemma.
\end{proof}

Scaling shows that \eqref{eq:elliptic} implies the
 $h$ dependent coercivity of Theorem \ref{thm:coercivity1}.
\end{proof}

\section{Stability}

\begin{theorem} 
\label{thm:stability1}
Let $P^h$, be the operator  \eqref{eq:maxwelldiffractive} with 
$\epsilon^h, \partial_t\epsilon^h,\mu^h,\partial_t\mu^h  , M^h\in C^\infty_h( \RR^{1+3} )$.
If $f,g\in H^\infty(\R^3):=\cap_sH^s(\RR^3)$ then there is a unique family of solutions
$u^h=(E^h,B^h)\in C^\infty(\R\,;\,H^\infty(\R^3))$ to 
$P^hu^h=0$, 
with $u^h(0) = (f,g)$.    For each $ m\in \N$ there is a constant $c(m)$ 
so that with 
$$
C(m,h)\ :=\  \sum_{|\alpha|\le m} \,\big\|
(h\partial_{t,x})^\alpha
(\partial_t\epsilon^h, \partial_t\mu^h , M^h)
\big\|_{ L^\infty( \RR^{1+3}  ) }
$$
one has
\begin{equation}
\label{eq:derivest}
\big\|   u^h(t)  \big\|_{m,h}
\ \le \
c(m)\,
e^{c(m)\, C(m,h)\, t}\
\big\|   u^h(0)  \big\|_{m,h}
\,.
\end{equation}
\end{theorem}

\begin{remark}
\label{rmk:divergence}
If $M^h=0$, 
one has
$
\partial_t\big(\div\big(\epsilon^h(t,x)\, E^h\big)\big) = 0$
and
$
\partial_t\big(\div\big(\mu^h(t,x)\, B^h\big)\big)  =0$.
In particular one has
\begin{equation}
\label{eq:div4}
\div\big(\epsilon^h(t,x)\, E^h\big)\ =\ 0,
\qquad
\div\big(\mu^h(t,x)\, B^h\big) \ =\ 0\,,
\end{equation}
as soon as these identities hold at $t=0$.
\end{remark}

\begin{remark}\label{rem:diffrsize}
For the analysis on the diffractive scale, the theorem is 
applied with $C(m,h)=\Op(h)$ as $h\to 0$.  In that case
one has uniform bounds for $t=\Op(1/h)$.
\end{remark}

\begin{remark}
This careful accounting of derivatives 
in Theorem \ref{thm:stability1}
is at the heart of  extenting the results of this paper to
 equations
whose coefficients are only finitely differentiable.
\end{remark}

\begin{proof} 
The existence for fixed $h$ is classical.
The estimate for $m=0$ follows from 
Gronwall's inequality
together with
$$
\begin{aligned}
\partial_t \int_{\RR^3} \left(
\langle \epsilon^h \, E^h\,,\, E^h\rangle
 +  
\langle \mu^h \, B^h\,,\, B^h\rangle\right) dx
 = &
- \int_{\RR^3} \left(  
\langle \partial_t\epsilon^h \, E^h\,,\, E^h\rangle
 +
\langle \partial_t\mu^h \, B^h\,,\, B^h\rangle\right) dx \\
& +
2 \int_{\RR^3} \langle u^h, M^hu^h\rangle \, dx \\
 \le C
\Big(
\|\partial_t\epsilon^h(t),\partial_t\mu^h(t)\|_{L^\infty(\RR^3)}
\ & +\ 
\|M^h(t)\|_{L^\infty(\RR^3)}
\Big)
\|(E^h,B^h)(t)\|_{L^2(\RR^3)}\,.
\end{aligned}
$$ 
The proof for arbitrary $m$ is by induction.  
Assume the estimate proved 
for $m$.   Denote by $K^h(t,s)$ the evolution operator associated
to the equation $P^h w=0$.  Precisely, for a distribution 
$g\in {\mathcal D}^\prime(\RR^3)$,
$w(t)=K^h(t,s)g$ is the unique solution of the Cauchy problem
$$
P^h\, w\ = \ 0\,,
\qquad
w\big|_{t=s}\ =\ g\,.
$$
The inductive hypothesis is a bound
$$
\|K^h(t,s)g\|_{ m,h }
\ \le \
c(m)\, e^{  c(m) C(m,h) |t-s| }\,
\|  g  \|_{ H^m_h(\RR^3) }
\,.
$$
The Duhamel relation
$$
v^h(t) \ =\ K^h(t,0)v^h(0) 
\ +\ 
\int_0^t K^h(t,s)\ P^h(v^h)(s)\ ds
$$
then yields the estimate
\begin{equation*}
\|  v^h(t) \|_{m,h}
\le
c(m)
\Big(
e^{c(m)C(m,h) t }
\|  v^h(0) \|_{m,h}
 + 
\int_0^t e^{c(m)C(m,h)(t-s)}
\ 
\| P^h(v^h)(s) \|_{m,h}
\ ds
\Big)
.
\end{equation*}
The key point is that one cannot simply 
apply the estimate
at level $m$ to $v^h:=h\partial u^h$.
Doing so yields
\begin{equation*}
\label{eq:duhamel}
\| h\partial u^h(t) \|_{m,h}
\le
c(m)
\Big(
e^{c(m)C(m,h) t }
\| h\partial u^h(0) \|_{m,h}
 + 
\int_0^t e^{c(m)C(m,h)(t-s)}
\ 
\| P^h(h\partial u^h)(s) \|_{m,h}
\ ds
\Big)
.
\end{equation*}
Write $P^h (h\partial)=(h\partial)P^h + [P^h,h\partial]$.
The first term 
vanishes when applied to $u$.
The commutator  $[P^h,h\partial]$ is 
$$
h\partial_t
\begin{pmatrix}
\partial \epsilon^h & 0 
\cr
0 & \partial\mu^h
\end{pmatrix}
 \ +\ 
 h\,\partial M^h\,.
$$
If $\partial$ were a derivative with respect to $x$
then $\partial_x\{\epsilon^h, \mu^h\}\sim 1/h$
leading to an unacceptably large contribution.
Instead of estimating all derivatives we estimate
only $\partial_t$, $\div\{\epsilon^h E^h,  \mu^hB^h\}$, and $\curl$.
The time derivatives of $\epsilon^h,\mu^h$
are bounded.   The
divergence and curl play a special role
in Maxwell's equations.  The remaining
spatial derivatives are recovered using coercivity.
\footnote{The use of  $\partial_t, \div ,\curl$
is reminiscent of the use of $\partial_t +\bfv\partial_x, \div,\curl$,
and tangential derivatives
for the inviscid compressible Euler equations
in \cite{nishidarauch}.
}

For $\partial_t$
compute
$$
[P^h,h\partial_t]
\ =\
\begin{pmatrix}
\partial_t \epsilon^h & 0
\cr
0 & \partial_t\mu^h
\end{pmatrix}
h\partial_t
 \ +\
 h\,\partial_t M^h\,.
$$
Estimate
$$
\| P^h(h\partial_t u^h)(s) \|_{m,h}
\ \le\
C(m+1,h)\,
\|u^h(s)\|_{m+1,h}
\,.
$$
The Duhamel estimate yields
\begin{equation}
\label{eq:duhamel2}
\| h\partial_t u^h(t) \|_{m,h}
\le
c(m)
\Big(
e^{c(m)C(m,h) t }
\| u^h(0) \|_{m+1,h}
 +
\int_0^t e^{c(m)C(m,h)(t-s)}
\
C(m+1,h)
\| u^h(s) \|_{m+1,h}
\ ds
\Big)
.
\end{equation}

Write the Maxwell equations as
$$
\big(
h\,\curl B^h\,,\,
-h\,\curl E^h
\big)
\ =\
h\partial_t (\epsilon^h\, E^h\,,\,
\mu^h\, B^h)  + h\,M^h\, u^h\,.
$$
Therefore
\begin{equation*}
\big\|
h\,\curl E^h(t)\,,\,
h\,\curl B^h(t)
\big\|_{m,h}
\ \le \
\Big(
{\rm r.h.s.\  of\  }
\eqref{eq:duhamel2}
+ \   \|  h\,  u^h(t) \|_{m,h}
\Big)
C(m,h)
\,.
\end{equation*}
 Use the
fundamental theorem of calculus  to estimate
$$
\|  h\, u^h(t)\ \|_{m,h} 
\le 
\|  h\, u^h(0)\ \|_{m,h} 
 + 
\int_0^t
\|  h\partial_t u^h(s) \|_{m,h}\ ds\,.
$$
Wasting a derivative in the first term yields 
$$
\le 
\| u^h(0)\ \|_{m+1,h} 
 + 
\int_0^t
\|   u^h(s) \|_{m+1,h}\ ds
.
$$
Combining the last four assertions yields
\begin{equation}
\label{eq:curl}
\big\|
h\,\curl E^h(t)\,,\,
h\,\curl B^h(t)
\big\|_{m,h}
\ \le \
{\rm r.h.s.\  of\  }   
\eqref{eq:duhamel2}
\,.
\end{equation}

Next estimate $\|h\, \div\, \epsilon^h E^h,h\,\div\, \mu^h B^h\|_{m,h}$.   Write
$M^h$ in block form with $3\times 3$ blocks
$$
M^h\ =\ 
\begin{pmatrix}
M^h_{11} & M^h_{12}
\cr
M^h_{21} & M^h_{22}
\end{pmatrix}\,.
$$
Therefore
$$
\partial_t \big(\div (\epsilon^h E^h)\big)
\ =\
\div (\epsilon^h E^h)_t \ =\ 
- \div (M^h_{11} E^h + M^h_{12}B^h)\,.
$$
Integrate to find
\begin{equation*}
\begin{aligned}
\| h\, \div (\epsilon^h E^h)(t)\|_{m,h}
&\ \le \
\| h\, \div (\epsilon^h E^h)(0)\|_{m,h}
\ +\ 
\int_0^t\|M^h_{11} E^h(s) + M^h_{12}B^h(s)\|_{m,h}\ ds
\cr
&\ \le \
C(m,h) \,  \| u^h(0)\|_{m+1,h}
\ +\ 
\int_0^t  C(m,h) \|u^h(s)\|_{m,h}\ ds
\,.
\end{aligned}
\end{equation*}
Performing the  analogous estimate for $h\,\div (\mu^h B^h)$ and replacing 
$\|  \   \|_{m,h}$ on the right by the larger
$\|  \  \|_{m+1,h}$
 yields
 \begin{equation}
\label{eq:divergence}
\| h\, \div (\epsilon^h E^h)(t) \,,\,  
h\,\div(\mu^h B^h) 
\|_{m,h}
\ \le \
{\rm r.h.s.\  of\  }   
\eqref{eq:duhamel2}
\,.
\end{equation}

Combining \eqref{eq:duhamel2}, \eqref{eq:curl} and \eqref{eq:divergence}, the inductive 
hypothesis, and 
 the coercivity estimate from Theorem \ref{thm:coercivity1}
yields
\begin{equation}
\label{eq:duhamel3}
\| u^h(t) \|_{m+1,h}
\le
c(m)
\Big(
\| u^h(0) \|_{m+1,h}
 + 
\int_0^t e^{c(m)C(m,h)(t-s)}
\,
C(m+1,h)
\| u^h(s) \|_{m+1,h}
\, ds
\Big)
.
\end{equation}
The theorem follows from Gronwall's
inequality.
\end{proof}

\section{Stationary solutions}

When $\epsilon$ and $\mu$ depend only on $x$, 
there is a conserved $L^2$ norm for the solution of 
\eqref{eq:Maxwell},
\begin{equation}
\label{eq:L2cons}
\partial_t \int_{\R^3} 
\langle E\,,\, \epsilon(x)\,E \rangle \ +\ 
\langle B\,,\, \mu(x) \, B\rangle \ dx
\  =\ 0\,.
\end{equation}
Equations \eqref{eq:Maxwell} have an infinite dimensional
space of stationary solutions.
The set of functions satisfying 
$\div\,\epsilon(x)E =
\div\,\mu(x)B  = 0$
 is invariant and orthogonal in the 
conserved
$L^2$ scalar product to the stationary solutions.

To prove this assertion we use the Fourier Transform. 
For any $u\in L^2(\R^3)$, we have
$$
\hat u(\xi) \ =\ 
(2\pi)^{-3/2}\,
\int_{\R^3} e^{-ix.\xi}\ u(x)\ dx,
\qquad
u(x)
\ =\ 
(2\pi)^{-3/2}\,
\int_{\R^3} e^{ix.\xi}\ \hat u(\xi)\ d\xi\,.
$$
The Fourier Transform is unitary on $L^2(\R^3, dx)$.

\begin{theorem}
\label{thm:stationary1}
When $\epsilon(x),\mu(x)$ depend only on $x$,
$u(x)=(E(x),B(x))\in L^2(\R^3)$ is a 
stationary solution of the dynamic Maxwell
equations \eqref{eq:Maxwell}
if and only if $\curl E=\curl B=0$.     The orthogonal complement
of these data in $L^2_{\epsilon,\mu}(\R^3)$ normed by
$\int \langle\epsilon E,E\rangle + \langle\mu B,B\rangle \, dx$ is invariant
under the flow and consists exactly of the solutions
satisfying \eqref{eq:div4}.
The fields $(\grad\, \phi, \grad\,\psi)$ with 
$\phi,\psi \in H^1(\R^3)$ are $L^2(\R^3)$-dense in the
stationary solutions.
\end{theorem}

\begin{proof}  The first assertion is obvious.

The orthogonal complement of the stationary solutions is invariant
because the evolution is unitary.

Next prove the density of gradients.
The field $E$ satisfies
 $\curl E=0$ if and only if $\xi\wedge \widehat E(\xi)=0$, that is 
the Fourier Transform $\widehat E(\xi)$ is parallel to $\xi$ for almost all $\xi$.
For $\xi \ne 0$  $\widehat E = \xi f(\xi)$ uniquely defines the scalar valued
$f(\xi)$.  Choose a smooth cutoff function $0\le \chi\le 1$ vanishing on a neighborhood
of $\xi=0$ and identically equal to 1 outside a compact set.
Then $E$ is the $L^2$-limit as $n$ tends to infinity of the field
with Fourier transform equal to $\chi(n\xi)\, \xi\,f(\xi)$. 
Define $\widehat \phi^n=\chi(n\xi)f$ so $\phi\in H^1(\R^3)$
and $\grad\, \phi^n \to E$ in $L^2$.

Using the density of the Schwartz space ${\mathcal S}(\R^3)$ in $H^1(\R^3)$, shows that
a vector $E$ is orthogonal to the stationary states if and only if 
$$
\int_{\R^3}  \langle \epsilon E, \grad \phi \rangle \ dx \ =\ 0,
\quad
\text{for all}
\quad
\phi\in {\mathcal S}(\R^3)\,.
$$
This is the definition of 
$\div (\epsilon E)=0$ in the sense of tempered 
distributions.  
\end{proof}

\section{The theory of Floquet and Bloch}\label{sec:Bloch-theory}

\subsection{Bloch Transform}
\label{blockspec}

We recall the essentials
of the method of Floquet and Bloch
(see for example 
\cite{flo}, 
\cite{bloch},
\cite{bril}, 
\cite{reedsimon},
\cite{wilcox},
\cite{blp}).  

Write each $\xi\in \RR^3$ as $\bfn +\theta$ with uniquely
determined
$\bfn \in\ZZ^3$ and $\theta\in [0,1[^3$.
Expressing $u(x)$ in terms of its Fourier transform, 
$\hat u(\xi)$, yields
\begin{equation}
\label{eq:block1}
u(x)
 \ = \
 (2\pi)^{-3/2}\,
 \int_{[0,1[^3}
 e^{ i\theta .x}\
 \Big(
 \sum_{\bfn\in \ZZ^3}
 e^{ i \bfn .x} \hat u(\theta +\bfn)\Big)
 \ d\theta\,.
\end{equation}
The parentheses enclose the Fourier
series expansion of a function periodic in $x$ with period $2\pi$.  
Considered as a function of $x$, the integrand is $\theta$-periodic
in the sense of Definition \ref{def:thetaper}. 
Identity \eqref{eq:block1} decomposes $L^2(\RR^3)$
as the direct integral over $\theta$ of the 
Hilbert spaces of  $\theta$-periodic functions
belonging to $L^2_{\rm loc}(\RR^3)$.

\begin{proposition}
\label{prop:blocktransform}
The map that associates to $u$ its Bloch wave expansion $u_\theta$
\begin{equation}
\label{eq:decomp}
L^2(\R^3)\ni u
\quad 
\mapsto
\quad
u_\theta(x) \ :=\ 
 (2\pi)^{-3/2}\,
  e^{ i\theta.x}\
 \Big(
 \sum_{\bfn\in \ZZ^3}
 e^{ i \bfn.x} \hat u(\theta +\bfn)\Big)
 \ \in \ 
L^2(\T^3_\theta)
\end{equation}
yields a unitary decomposition of $L^2(\R^3)$
as the direct
integral over $\theta\in [0,1[^3$ of 
$
L^2(\T^3_\theta)
$.
 The inverse is given by
$$
u(x) \ =\ 
\int_{[0,1[^3} u_\theta(x)\ d\theta
\qquad
{\rm with}
\qquad
\|u\|_{L^2(\R^3)}^2
\ =\ \int_{[0,1[^3}
 \|u_\theta\|_{L^2(\T^3)}^2
 \ d\theta
\,.
$$
The map $u\mapsto e^{-i\theta .x} u_\theta$ 
is a unitary map $ L^2(\R^3)
\to
 L^2([0,1[^3\,;\,L^2(\T^3))
$.
\end{proposition}

\begin{remark}
\label{rem:reduction}
The partial derivatives and $2\pi$-translates
 of $\theta$-periodic functions
are $\theta$-periodic and the product of a $\theta$-periodic
function by a $2\pi$-periodic function is $\theta$-periodic.
Therefore partial differential operators with
$2\pi$-periodic coefficients  map 
$\theta$-periodic functions to themselves.
The Bloch decomposition
reduces these operators.
\end{remark}

\subsection{Maxwell's equations}

The method of Floquet-Bloch 
applies to Maxwell's equations
(see for example \cite{blp} and
\cite{sjoberg}).
The delicate part for us is the infinite dimensional kernel
and the degeneration of the 
coercivity estimates as $h\to0$.
From here to the end of this section
the Bloch strategy  
is  used to analyse Maxwell's
equations
\begin{equation}
\label{eq:defgen}
\begin{pmatrix}
\epsilon_0(x)
&0
\cr
0&\mu_0(x)
\end{pmatrix}
\partial_t \ -\ G\,,
\qquad
G\ :=\ \begin{pmatrix}
0 & \curl
\cr
-\curl & 0
\end{pmatrix}
\end{equation}
{\it in the case of periodic} $\epsilon(x)$ and
$\mu(x)$.
It suffices to analyse its action
as a map on
$L^2(\T^3_\theta)$

\begin{proposition}
\label{prop:antisa}
If $A_j$ are symmetric matrices then the 
operator $L=\sum_jA_j\partial_j$  
satisfies
\begin{equation}
\label{eq:antisymmetry}
\forall u,v\in L^2(\T^3_\theta)\cap C^\infty
\qquad
\langle Lu,v\rangle \  =\ - \langle u,Lv \rangle
\,.
\end{equation}
Denote by $\overline L$ the closure of the operator so defined
and by $L^*$ the Hilbert space adjoint.
Then $\overline L$ is antiselfadjoint
with domain equal to the set of
$u\in L^2(\T^3_\theta)$  so that 
$\sum_j A_j\partial_ju\in L^2(\T^3_\theta)$ 
 in the sense
of distributions.
\end{proposition}

\begin{proof}   
To prove \eqref{eq:antisymmetry},
write $u=e^{i\theta .x} \tilde u$ and similarly $v$ 
with periodic $\tilde u, \tilde v$. Then,
$$
\begin{aligned}
\langle Lu,v\rangle
&\ =\ 
\int _{[0,2\pi[^3}
 \langle A_j \partial_j u\,,\, v \rangle \ dx
\ =\ 
\int _{[0,2\pi[^3} \langle A_j\partial_j (e^{i\theta .x}\tilde u)\,,\, e^{i\theta .x}\tilde v \rangle \ dx
\cr
&\ =\
\int _{[0,2\pi[^3}
 \langle e^{i\theta .x}A_j(\partial_j+i\theta_j) \tilde u\,,\, e^{i\theta .x}\tilde v \rangle \ dx
\ =\
\int _{[0,2\pi[^3} \langle A_j (\partial_j+i\theta_j)  \tilde u\,,\, \tilde v \rangle \ dx
\cr
&\ =\ 
-
\int _{[0,2\pi[^3} \langle \tilde u\,,\, A_j(\partial_j+i\theta_j) \tilde v  \rangle \ dx
\,,
\end{aligned}
$$
the last step by integration by parts and periodic boundary conditions.  

Identity \eqref{eq:antisymmetry}  implies
that $L^*\supset -L$  in the sense that the left hand side
is an extension of the right.   The definition of distribution derivative
implies that $L^*u=f\in L^2(\T^3_\theta)$ if and only if 
$-(\sum_jA_j\partial_j)u=f$ in the sense of distributions.

In that case denote by $J_\delta$ a standard mollifier.
Since $\theta$-periodic functions are invariant under
translations, $J_\delta u\in C^\infty\cap L^2(\T^3_\theta)$.
Since $J_\delta$ commutes with $A_j\partial_j$ so
$u^{\delta} :=J_\delta u$ satisfies $-Lu^\delta = J^\delta f$.
Passing to the limit shows that $u$ belongs to the 
domain of 
$\overline L$ and $\overline Lu=-f$.   Thus $L^*\subset - \overline L$.
\end{proof}

The spaces $L^2(\T^3_\theta)$ depend on $\theta$.
The following proposition allows 
one to apply standard results in perturbation theory.

\begin{proposition}
\label{prop:analyticfamily}
The unitary map $L^2(\T^3)\ni v\mapsto e^{i\,\theta .x}\,v\in L^2(\T^3_\theta)$ 
intertwines the operators
$$
\begin{pmatrix}
0 & \partial_x\wedge
\cr
-\partial_x\wedge & 0
\end{pmatrix}
\,,
\qquad
{\rm and}
\qquad
\begin{pmatrix}
0 & (\partial_x+i\theta)\wedge
\cr
-(\partial_x+i\theta)\wedge & 0
\end{pmatrix}\,.
$$
The former acts on $L^2(\T^3_\theta)$ and the latter
on the $\theta$ independent space $L^2(\T^3)$. 
The latter family of operators depends analytically on $\theta$.
\end{proposition}

\begin{proof} 
The unitary map commutes with multiplication operators
and intertwines
the antiselfadjoint operators $\partial_j$  on 
$L^2(\T^3_\theta)$ and
$\partial_j +i\theta_j$
on $L^2(\T^3)$.  This yields
the desired result.
\end{proof}

The straight forward proof of the next result is omitted.

\begin{proposition}
\label{prop:stationary2}
A function $u(x)=(E(x),B(x))\in L^2(\R^3)$ is a stationary
solution
of \eqref{eq:Maxwell}  if and only if its Bloch wave expansion 
$(E_\theta,B_\theta)\in L^2(\T^3_\theta)$ satisfies
$G\,(E_\theta,B_\theta)=0$ for almost all $\theta\in [0,1[^3$ with
$G$ from \eqref{eq:defgen}.  A function $u$ is 
in the invariant space of functions orthogonal
in $L^2_{\epsilon_0,\mu_0}(\T^3_\theta)$
to these stationary solutions if and only if its expansion satisfies
$\div (\epsilon_0 E_\theta)=\div(\mu_0 B_\theta)=0$ 
in the sense of distributions.
\end{proposition}

\subsection{Bloch spectral theory}
Consider  fixed periodic
$\epsilon_0(x), \mu_0(x)\in C^\infty(\TT^3)$.
The function
$$
u(t,x)
\ :=\
e^{\lambda t}\
\big(
E(x)\,,\, B(x)\big)
$$
satisfies Maxwell's  equations \eqref{eq:Maxwell}
if and only if
\begin{equation}
\label{eq:evalue2}
\lambda\,\begin{pmatrix}
\epsilon_0(x) & 0
\cr
0 & \mu_0(x)
\end{pmatrix}
\begin{pmatrix}
E(x)\cr B(x)
\end{pmatrix} 
\ =\ 
\begin{pmatrix}
 0 & \curl 
\cr
-\curl & 0
\end{pmatrix}
\begin{pmatrix}
E\cr B
\end{pmatrix}\,.
\end{equation}
In the same way a $\theta$-periodic $u=e^{\lambda t}(E(x),B(x))$
is a $\theta$-periodic solution of Maxwell's equations if and
only if $E,B$ is a solution of \eqref{eq:evalue2} and 
$(E\,,\, B)$ is $\theta$-periodic. 
The function $u$ is then a solution of \eqref{eq:Maxw3}.

When each of $\epsilon_0$ and $\mu_0$ 
is a positive  constant times  the identity
the
change of variable $\widetilde E,\widetilde B = \epsilon_0^{1/2} E, 
\mu_0^{1/2}B$ reduces to the case $\epsilon_0=\mu_0=I$.  For that
case the eigenvalue problem is  exactly solved
in the next example.

\begin{example} 
\label{important-example}
In case $\epsilon_0=\mu_0=I$ the problem
is translation invariant.   Denote by
$T_\ell$ the translation  operator $u(x)\mapsto u(x-\ell)$
acting on $L^2(\T^3_\theta)$.
The antiselfadjoint  $G$ commutes with $T_\ell$ so
the eigenspaces of $T_\ell$ are invariant
by $G$.

For given $k\in \ZZ^3$ denote by $E_k$ the subspace consisting of exponentials
$e^{i k .x} e^{i \theta .x} (\bfe,\bfb)$ when $\bfe, \bfb$ run in $\RR^3$.
$E_k$ consists of eigenvectors of $T_\ell$
with eigenvalue $e^{i(k+\theta) .\ell}$.
Choose the vector $\ell$ so that the
$\ell_j/2\pi$ are rationally independent.  Then the eigenvalues
for distinct $k$ are distinct, so the spectral decomposition of
$T_\ell$ is
$L^2(\T^3_\theta) = \oplus_\perp E_k$.

It follows that $G(E_k)\subset E_k$ for all
$k\in \ZZ^3$.
It suffices to
diagonalize the restriction of $G$ to each $E_k$.
Compute
$$
G\, \bigg(
e^{i\theta .x}e^{i k .x}
\begin{pmatrix}
\bfe
\cr
\bfb
\end{pmatrix}
\bigg)
\ =\
\begin{pmatrix}
 0 & \curl
\cr
-\curl & 0
\end{pmatrix}
\bigg(
e^{i\theta .x }\begin{pmatrix}
e^{ik .x}\bfe
\cr
e^{ik .x}\bfb
\end{pmatrix}
\bigg)
\ =\
e^{i\theta  .x}
e^{i k .x}
\begin{pmatrix}
i(k+\theta)\wedge \bfb
\cr
-i(k+\theta)\wedge \bfe
\end{pmatrix}\,.
$$
To have eigenvalue $\lambda=i\omega$ it is necessary and sufficient
that $\omega$ is an eigenvalue of  $G_0\in {\rm Hom}(\CC^6)$ 
\begin{equation}
\label{eq:evcond}
G_0\, 
\begin{pmatrix}
\bfe
\cr
\bfb
\end{pmatrix}
\ :=\ 
\begin{pmatrix}
(k +\theta) \wedge \bfb
\cr
- (k+\theta)\wedge \bfe
\end{pmatrix} 
\ =\ \omega 
\begin{pmatrix}
\bfe
\cr
\bfb
\end{pmatrix} 
\,,
\end{equation}
One has eigenvalue $0$ if and only if
both $\bfe$ and $\bfb$ are parallel to $k+\theta$
This kernel has dimension
equal to two.

The orthogonal space has dimension $4$ and consists
of vectors with both $\bfe$ and $\bfb$ orthogonal to $k+\theta$.
Since $(k+\theta)\perp \bfb$,
$(k+\theta) \wedge(k+\theta)\wedge \bfb = -|k+\theta |^2 \bfb$.
Using \eqref{eq:evcond} compute
$$
-|k+\theta |^2 \bfb
=
(k+\theta)\wedge(k+\theta)\wedge \bfb
=
(k+\theta)\wedge(\omega\bfe)
=
-\omega^2\bfb\,.
$$
Therefore for $(k+\theta)\ne 0$ there are two roots
$\omega=\pm |k+\theta|$.   Each has a two dimensional
eigenspace generated by taking $\bfe\perp(k+\theta)$
and $\bfb=\mp (k+\theta)\wedge \bfe$.
\end{example}

Return next to the general case.

\begin{proposition}
\label{prop:thetaspec}
There is an infinite dimensional space
of $\theta$-periodic solutions of \eqref{eq:evalue2}
with  eigenvalue
$\lambda=0$ consisting of
$E$ and $B$ with vanishing {\rm curls}.
The  $L^2_{\epsilon_0,\mu_0}(\T^3_\theta)$-orthogonal 
complement to this kernel
 satisfy
 $\div(\epsilon_0 E)=\div
(\mu_0 B) =0$.
There is a constant $C$ independent of $\theta$
so that
$\theta$-periodic $E,B$
satisfy
\begin{equation}
\label{eq:thetacoercivity}
\begin{aligned}
\|E,B&\|_{H^1(\T_\theta^3)} 
\ \le\
C\Big(
\|\curl\,E\|_{L^2(\T^3_\theta)}
\ +\
\|\curl\,B\|_{L^2(\T_\theta^3)}\cr
&
\hskip.5cm
\ +\
\|\div(\epsilon_0(x) E)\|_{L^2(\T_\theta^3)}
\ +\
\|\div\,(\mu_0(x) B)\|_{L^2(\T_\theta^3)}
\ +\
\|E,B\|_{L^2(\T_\theta^3)}
\Big)\,.
\end{aligned}
\end{equation}
For each  $\theta$ the $\lambda=i\omega\ne 0$ for which
\eqref{eq:evalue2} has a nontrivial solution
is a discrete set in  $\RR\setminus \{0\}$.  
Each eigenspace
is a finite dimensional space of smooth functions.
 The value $0$ is not an accumulation point
of nonzero eigenvalues.  The $\omega$ are
not bounded above and are not bounded below.
\end{proposition}

\begin{proof}  The estimate is the key.  Write $u=e^{i\theta x}\widetilde u$
with periodic $\widetilde u$.  Expand the latter in a Fourier series.
The proof of
the Lemma \ref{lem:coercivity}
yields \eqref{eq:thetacoercivity}.

Proposition \ref{prop:stationary2} implies that
the stationary solutions are curl free and their orthogonal complement
is invariant under the flow by Maxwell's equations.
Also that the  complement 
consists of solutions satisfying
$\div(\epsilon_0(x) E)=\div
(\mu_0(x) B) =0$.

Decompose $L^2_{\eps_0,\mu_0}(\TT^3_\theta)$ 
as a Maxwell invariant direct sum of stationary and dynamic
states
$$
L^2_{\eps_0,\mu_0}(\TT^3_\theta)
\ =\ 
{\rm Ker}\, G
\ \oplus_\perp\, 
\caH_{\rm dyn}\,,
\qquad
\caH_{\rm dyn}
\ :=\
\big\{
E,B\,:\,\div(\eps_0(x)E)\ =\ \div\mu_0(x)B\ =\ 0\big\}
.
$$ 
The 
$L^2_{\eps_0,\mu_0}(\TT^3_\theta)$ 
unitary group $e^{tG}$ 
restricts to a unitary  group
on $\caH_{\rm dyn}$ whose anti selfadjoint
generator is the  restriction $G|_{\caH_{\rm dyn}}$.

Estimate \eqref{eq:thetacoercivity} implies that 
$(I + G|_{\caH_{\rm dyn}})^{-1}$ is  compact,
hence   has pure point spectrum tending to zero
and total multiplicity  in $\{|z|\ge \delta>0\}$ finite
for all $\delta > 0$.
Therefore the spectrum of $G|_{\caH_{\rm dyn}}$
is
pure point and the total multiplicity in any bounded
set is finite.

Commuting with derivatives yields an inductive proof of
an $H^s$ version of \eqref{eq:thetacoercivity} for $1\le s\in \NN$,
\begin{equation}
\label{eq:scoercivity}
\begin{aligned}
\|E,B&\|_{H^{s+1}(\T_\theta^3)} 
\ \le\
C(s)\Big(
\|\curl\,E\|_{H^s(\T^3_\theta)}
\ +\
\|\curl\,B\|_{H^s(\T_\theta^3)}\cr
&\ +\
\|\div(\epsilon_0(x) E)\|_{H^s(\T_\theta^3)}
\ +\
\|\div\,(\mu_0(x) B)\|_{H^s(\T_\theta^3)}
\ +\
\|E,B\|_{H^s(\T_\theta^3)}
\Big)\,.
\end{aligned}
\end{equation}
The smoothness
of eigenfunctions for eigenvalues $\lambda\ne 0$
follows.

It remains to show that the spectrum is unbounded
above and unbounded below.
Define $P$ to be the bounded
strictly positive selfadjoint multiplication
operator $P(E,B):=(\eps\, E, \mu\, B)$.  The eigenvalue equation is
$Gu=i\omega Pu$.  It holds if and only if $v=P^{1/2}u$ satisfies
$P^{-1/2}(G/i)P^{-1/2} v= \omega v$.  Need to  show that
 the eigenvalues are not bounded below and not bounded above.
The $\omega$
are bounded below (resp. above)  by $C\in \RR$  if and only if for all
   $v$ so that $P^{-1/2}v$ belongs to the domain of $G$,
   $$
   \langle
   P^{-1/2}(G/i)P^{-1/2} v\,,\, v
   \rangle
   \ \ge \
   C\big \langle v\,,\,v\big \rangle\,,
   \qquad
   ({\rm resp.} \ \le\ )\,.
   $$
   This holds if and only if for all $u$ belonging to the domain
   of $G$,
   $$
    \langle
   (G/i) u\,,\, u
    \rangle
   \ \ge \
   C\big \langle P^{1/2}u\,,\,
   P^{1/2}u
   \big \rangle\,,
    \qquad
   ({\rm resp.} \ \le\ )\,.
   $$
This holds if and only if the spectral problem with $\eps_0=\mu_0=I$ has
eigenvalues bounded below (resp. above).  Example \eqref{important-example} 
shows  by explicit computation
that for $\eps_0=\mu_0=I$
the spectrum is unbounded in both directions.  Thus the $\omega$
are not bounded below and not bounded above.
\end{proof}

Proposition \ref{prop:thetaspec}
implies that  the discrete spectrum at fixed $\theta$ consists of 
$\{0\}$ with infinite multiplicity and a discrete set of possibly multiple
eigenvalues $\lambda=i\omega$ that we label according to their distance
from the origin as in \eqref{eq:numbering}.
Away from eigenvalue crossings,
the functions $\theta\mapsto \omega_j(\theta)$
are real analytic.  
Rellich's theorem shows that 
away from the 
crossings the associated spectral projections $\Pi_j(\theta)$
are also real analytic in the sense that the unitary
map of Proposition \ref{prop:analyticfamily} intertwines
them with an analytic family acting on $L^2_{\epsilon,\mu}(\T^3)$ (see \cite{kato}).

For $\theta$ fixed and an eigenvalue $\omega(\theta)\ne 0$ there is a 
finite dimensional space of eigenfunctions $e^{i\theta .x} (E_\theta(x) , B_\theta(x))\in L^2(\T^3_\theta)$ 
and corresponding {\bf Bloch plane wave} solutions of \eqref{eq:Maxwell}
$$
e^{i\omega(\theta)t}\
e^{i\theta .x} \
(E_\theta(x) , B_\theta(x))\,.
$$

We assume the constant  multiplicity Hypothesis
\ref{hyp:constantmultiplicity}.

From the analytic dependence of the operators it follows that the 
$L^2(\T^3_\theta)$ orthogonal projection, $\Pi(\theta)\in {\rm Hom}\,(L^2(\T^3_\theta))$
onto the 
nullspace of 
$i\omega(\theta)A_0^0(y)-\sum_jA_j\partial_j$
 is analytic and in particular of constant rank.

\section{The purely periodic case}
Fix $\utheta$ and a locally constant multiplicity
eigenvalue $\omega(\theta)\ne 0$.  Denote by
$\KK(\theta)$ the kernel of $\LL(\omega(\theta),\theta,y,\partial_y) $
from Definitions \ref{def:L-P}.
If $\theta\mapsto e^{i\theta .x} \psi(x,\theta)$ is
a smooth function of $\theta$ on a neighborhood
of $\utheta$ with values in $\KK(\theta)$
then $\psi$ is periodic with period $2\pi$ in $x$
and smooth in its dependence on $x,\theta$ for
$\theta\approx \utheta$.
The function
$$
e^{i\omega(\theta)t}\ e^{i\theta .x}\ \psi(x,\theta)
$$
is a $\theta$-periodic solution
of Maxwell's equation
in the periodic medium $\epsilon(x),\mu(x)$.

Scaling the periodic structure to 
$\epsilon(x/h),\mu(x/h)$ yields the corresponding
rapidly oscillatory Bloch plane waves
$$
e^{ i\omega(\theta)t/h } \ e^{ i\theta .x/h } \ 
\psi(x/h\,,\,\theta)\,.
$$
For 
$a\in C^\infty_0(\R^3)$
and 
$h<<1$ the function $a((\theta-\utheta)/h)$
is supported in the domain of definition of $\omega(\theta)$.
Superposing nearby waves yields a  Bloch wave packet for $h<<1$
$$
\int_{[0,1[^3} a\Big( \frac{ \theta-\utheta }{ h } \Big)\
e^{i\omega(\theta)t/h} \ e^{i\theta .x/h}\ 
\psi(x/h\,,\,\theta)\ d\theta,
\qquad
a\in C^\infty_0(\R^3)\,.
$$
Letting $\zeta:=(\theta-\utheta)/h$ 
yields the
exact solutions
\begin{equation}
\label{2.2.2}
u^h\ =\ 
e^{ i \utheta .x/h}\
\int_{\RR^3}
\psi\Big(
\frac{x}{h}
\,,\,
\utheta + h\zeta
\Big)
\ 
e^{ i t \omega(\utheta +h\zeta)/h}
\
e^{ ix.\zeta}
\ 
a(\zeta)
\ d\zeta\,.
\end{equation}

\subsection{The geometric optics time scale $\bf t\sim 1$.}  
\label{sec:ppgop}

\begin{definition}
\label{def:DD}
Complementing Definition \ref{def:groupvelocity}
the corresponding transport operator
is defined by
$$
\DD(\partial_t,\partial_x)\ :=\ 
\partial_t \ +\  \V.\partial_x\,.
$$
The symbol is $\DD(\tau,k)=\tau+\V .k$.

\end{definition}

The Taylor series in $h$ of infinite and finite orders
respectively are,
\begin{equation}
\label{psiexp}
\psi(y, \utheta+h\zeta)
\ \sim\ 
\psi(y,\utheta) + \sum_{j\geq 1} h^j\ g_j(y,\zeta)\,,
\quad
\omega(\utheta+h \zeta)
\ =\ 
\omega(\utheta) -\V h \zeta + h^2k(h,\zeta)\,,
\end{equation}
where the sum on the right hand  side of $\sim$
is an asymptotic expansion as $h\to 0$, not a convergent
series.
Then,
\begin{equation}
\label{exponential}
\begin{aligned}
e^{ it\omega(\utheta+\e\zeta)/h}
&\ =\ 
e^{ it\omega(\utheta)/h}\
e^{- i t\V.\zeta}\
e^{ i (h t) k(h, \zeta)}\cr
&\ =\ 
e^{ it\omega(\utheta)/h}\
e^{- i t\V.\zeta}\
\Big(1 + \sum_{j\geq 1} (h t)^j k_j(h, \zeta)\Big)\,,
\end{aligned}
\end{equation}
where the last line uses a Taylor expansion of 
$s\mapsto e^{ i s k(h,\zeta)}$ about $s=0$.
Define
$$
v(x) \ :=\ \int 
\ 
e^{ ix.\zeta}
\ 
a(\zeta)
\ d\zeta\,.
$$
Use \eqref{psiexp} and \eqref{exponential}
in \eqref{2.2.2}  to find
\begin{equation}
\label{ppgop}
u^h
\ \sim\
e^{ i S/h}\
w(h,t,x,x/h)\,,
\qquad
S \ :=\ \omega(\utheta)t + x.\utheta\,,
\end{equation}
and
$$
w(h,t,x,y)
\ \sim\
w_0(t,x, y) + h w_1(t,x,y) + \cdots
$$
in the sense of Taylor series about $h=0$.  The 
leading term is
$$
w_0(t,x,y) \ =\ v(x-\V t)\, \psi(y,\utheta)\,.
$$
The velocity is $\V=-\nabla_\theta\omega(\utheta)$. It is not at all obvious from this definition
that $\V$ does not exceed the speed of light. 
As our media are anisotropic, the speed may depend
on direction.  We recall the algorithm giving such
anisotropic speeds (see \cite{leray}, 
\cite{jmrhj}, \cite{rauch}).  

Denote by $(\tau,\xi)$ the dual variable of $(t,y)$.
The characteristic 
polynomial of Maxwell's equations is 
$$
p(y,\tau,\xi) \ :=\ \det \left( \tau 
\begin{pmatrix}
\epsilon_0(y) & 0 \cr 
0 & \mu_0(y)
\end{pmatrix} + 
\begin{pmatrix}
0 & - \xi \wedge \cr 
\xi \wedge & 0
\end{pmatrix}
\right)
\,.
$$
Its roots $\tau(y,\xi)$ for $\xi$ real
define the characteristic variety.
Define extreme roots 
$\tau_{\rm max}(y,\xi)$ 
by
$$
 \tau_{\rm max}(y,\xi)
  \  :=\
 \max\big\{
 \tau\,:\, p(y,\tau,\xi)=0\big\}\,.
 $$
 with a similar defintion for $\tau_{\rm min}(y,\xi)$.
 The function $\tau_{\rm max}$ is positive homogeneous
 of degree one in $\xi$, and,
 \begin{equation}
 \label{eq:supportfunction}
 \tau_{\rm max}(y,-\xi) \ =\ 
 -\,
 \tau_{\rm min}(y, \xi)\,.
 \end{equation}
 The set of velocities that do not exceed the speed
 of propagation is a convex set of vectors $\bfv$
 whose support function equal 
 $\tau_{\rm max}(y,-\xi)$, equivalently
 $$
\cap_
{  \xi\in \RR^3\setminus \{0\} }
\
 \big\{ \bfv\,:\,
\xi.\bfv 
\ \le\
\tau_{\rm max}(y,-\xi)
\big\}\,.
$$

In case of constant and isotropic permittivities 
$\tau_{\rm max}(y , -\xi ) =|\xi| /\sqrt{\epsilon\mu}$
yielding the classic formula for the speed of light.

Define
$$
\tau_{\rm max}(\xi)
\ :=\
\max_{y\in \TT^3}\
\tau_{\rm max}(y,\xi)
\,.
$$
Then $\tau_{\rm max}(-\xi)$ is the largest speed limit
for 
$\xi.\bfv$.

\begin{theorem}
\label{thm.speed}
The  group velocity $\V$
from
Definition \ref{def:groupvelocity}
respects the maximum speed of propagation 
for each $\xi\ne 0$, precisely for all
$0\ne \xi\in \RR^3$, 
$  \xi .\V \le \tau_{\rm max}(-\xi)$.
\end{theorem}

\begin{proof}  
The Bloch spectral eigenvalue problem for periodic as opposed
to $\theta$-periodic functions is
\begin{equation}
\label{def.bscp}
i \, \omega(\theta) 
\begin{pmatrix}
\epsilon_0(y)&0 \cr
0& \mu_0(y)
\end{pmatrix} \Pi(\theta)
+
\left(
\begin{array}{cc}
0 & -(i\theta + \partial_y )\wedge \\
(i\theta + \partial_y )\wedge & 0
\end{array}
\right)
\Pi(\theta) = 0 
\end{equation}
where $\Pi(\theta)$ is the projector (of constant rank) 
on the eigensubspace of the eigenvalue $\omega(\theta)$. 
The constant multiplicity hypothesis implies that
the eigenvalue 
and the spectral projector are analytic in a vicinity of 
$\underline\theta$. 
Differentiating (\ref{def.bscp}) with respect to $\theta$ 
in the direction of a covector $\xi$ and multiplying on the left by 
$\Pi(\theta)$ yields
\begin{equation}
\label{def.bscp2}
\, \xi.\nabla_\theta\omega(\theta)
\ \Pi(\theta)
\begin{pmatrix}
\epsilon_0(y)&0 \cr
0& \mu_0(y)
\end{pmatrix} \Pi(\theta)
+ \Pi(\theta)
\left(
\begin{array}{cc}
0 & - \xi \wedge \\
 \xi\wedge & 0
\end{array}
\right)
\Pi(\theta) = 0\, .
\end{equation}
Take any eigenfunction $E(y) , B(y)$, normalized by 
\begin{equation}
\label{eq:normalization}
\int_{\TT^3}
 \langle
\begin{pmatrix}
\epsilon_0(y)&0 \cr
0& \mu_0(y)
\end{pmatrix}
\begin{pmatrix}
E(y) \cr 
B(y)
\end{pmatrix} 
,
\begin{pmatrix}
E(y) \cr 
B(y)
\end{pmatrix}
 \rangle
\
dy = 1 \,.
\end{equation}
The quadratic 
form associated to (\ref{def.bscp2}) yields
\begin{equation}
\label{eq:xiomega}
\xi.\nabla_\theta\omega(\theta)  \ +\
 \int_{\TT^3}
 \langle 
\left(
\begin{array}{cc}
0 & -\xi \wedge \\
\xi\wedge & 0
\end{array}
\right)
\begin{pmatrix}
E(y) \cr 
B(y)
\end{pmatrix} 
,
\begin{pmatrix}
E(y) \cr 
B(y)
\end{pmatrix}
 \rangle
\ dy
\ =\ 
0\, .
\end{equation}
The $-\tau(y,\xi)$ are the
eigenvalues with respect to the postitive
definite matrix ${\rm diag}\,\{\eps_0(y)\,,\,\mu_0(y)\}$ so the 
min-max characterization implies that for each $y$
$$
\begin{aligned}
\langle
\left(
\begin{array}{cc}
0 & -\xi \wedge
 \\
\xi\wedge & 0
\end{array}
\right)
\begin{pmatrix}
E(y) \cr 
B(y)
\end{pmatrix} 
,
\begin{pmatrix}
E(y) \cr 
B(y)
\end{pmatrix}
\rangle
\ &\le\
-\,\tau_{\rm min}  (y,\xi)\,
\langle
\begin{pmatrix}
\epsilon_0(y)&0 \cr
0& \mu_0(y)
\end{pmatrix}
\begin{pmatrix}
E(y) \cr 
B(y)
\end{pmatrix} 
,
\begin{pmatrix}
E(y) \cr 
B(y)
\end{pmatrix}
\rangle
\cr
\ &\le\
-\,\tau_{\rm min}  (\xi)\,
\langle
\begin{pmatrix}
\epsilon_0(y)&0 \cr
0& \mu_0(y)
\end{pmatrix}
\begin{pmatrix}
E(y) \cr 
B(y)
\end{pmatrix} 
,
\begin{pmatrix}
E(y) \cr 
B(y)
\end{pmatrix} 
\rangle
\,.
\end{aligned}
$$
Integrating and using
\eqref{eq:normalization} and
\eqref{eq:xiomega} 
proves 
$$
\xi.\nabla_\theta \omega(\theta)
\ -\
 \tau_{\rm min}(\xi)
 \  \ge\
  0
\,,
$$
which is the desired relation since
$\V = -\nabla_\theta\omega(\theta)$.
\end{proof}

An analogous result to Theorem \ref{thm.speed} is proved in \cite{apr3}, where a bound on the 
group velocity for scalar wave equations is given.

\subsection{The diffractive time scale $\bf t\sim 1/h$.}
\label{sec:ppdiffrgop}

In (\ref{exponential}), the expansion
parameter is $h  t$ so when $h t$ is not small,
 the approximation is not appropriate.  For the diffractive
scale 
$h t\sim 1$, one needs a refinement.
Take the next term in the Taylor expansion in
the exponent.
Denote by $q$ the symmetric quadratic 
expression
\begin{equation}
\label{2.2.4}
q(\zeta,\zeta)  \ := \ \sum_{i,j=1}^3\,
\frac{\partial^2 \omega(\utheta)}{\partial \theta_i\partial \theta_j}
\ \zeta_i\,\zeta_j\,.
\end{equation}
Then,
$$
\omega(\utheta + \e\zeta)
\ =\ 
\omega(\utheta)
\ -\ 
h \V .\zeta
\ +\
h^2 q(\zeta,\zeta)/2
\ +\ 
h^3 \sum_{j\geq0} h^j \beta_j(\zeta)\,,
$$
and,
$$
e^{ i\omega(\utheta+h\zeta)t/h}
\ =\ 
e^{ i\omega(\utheta)t/h}\
e^{-  i t\V.\zeta}\
e^{ i h  tq(\zeta,\zeta)/2}\
e^{ i h (h t) \sum_{j\geq0} h^j \beta_j(\zeta)}\,.
$$
Introduce the slow time $\caT =h t$.
The 
exact solution has the form
$$
e^{2\pi iS/h}\ \widetilde W(h, h t, x-\V t,x/h)\,,
\qquad
S = \omega(\utheta) t + \utheta.x\,,
$$
$$
\widetilde W(h , \caT, x, y) 
\ :=\ 
\int
\psi(y,\utheta+h \zeta)\  
e^{i \caT q(\zeta,\zeta)/2}\ 
e^{ i h \caT \sum_{j\geq0} h^j \beta_j(\zeta)}\ 
e^{ i x.\zeta}\
a(\zeta)\
d\zeta\,.
$$
Taylor expansion in $h$ yields
\begin{equation}
\label{exponexp}
e^{ i h \caT \sum_{j\geq0} h^j \beta_j(\zeta)}
\ =\ 
\Big( 1 + \sum_{j\geq 1} h^j\,r_j(\caT,\zeta) \Big)\,.
\end{equation}
Using  \eqref{psiexp} and \eqref{exponexp}
in the definition of $\widetilde W$ shows that 
\begin{equation}
\label{wsim}
\widetilde W(h,\caT,x,y)
\ \sim\ 
\sum_{j\geq 0} \ h^j\, \widetilde w_j(\caT,x,y)\,,
\end{equation}
with
\begin{equation}
\label{firstterm}
\widetilde w_0(\caT,x,y)
\ =\ 
\psi(y,\utheta)\ 
\int
   e^{ i \caT\,q(\zeta,\zeta)/2}
\ 
e^{ ix.\zeta}\
a(\zeta)
\ d\zeta
\,.
\end{equation}
This shows that the solution has an asymptotic expansion of the form
$$
e^{ iS/h}\ \widetilde W(h ,h t, x-\VV t, x/h)\,,
$$
with $\widetilde W$ satisfying \eqref{wsim}.

Equation 
\eqref{firstterm} implies that 
$\widetilde w_0$ is a tempered solutions (with values in $\KK$) 
of the Schr\"odinger equation
\begin{equation}
\label{ppschrod}
 i\, \partial_\caT\widetilde w_0 \ -\  
 \frac{1}{2}\,
  \partial_\theta^2\omega(\utheta)
\big( \partial_x,\partial_x \big)\,
 \widetilde w_0
\ =\ 
0\,.
\end{equation}
Though the function  $\widetilde w_0$
takes values in the finite dimensional
space $\KK$ the equation \eqref{ppschrod}
is scalar.   The constant rank hypothesis
is crucial here.

\section{Bloch wave packets on a modulated background and $t=\Op(1)$}
\label{sec:gopBloch}

This section considers solutions of the Maxwell's equations
for times $t=\Op(1)$.  
This time scale
is an essential first
step in treating the diffractive case.  
In order that the asymptotic description be nontrivial 
we
allow lower order terms in the equations.
In particular this 
includes  the case
of a possibly conducting medium with Ohm's law dissipation,
\begin{equation}
\label{eq:ohm}
\epsilon (x,x/h)\, \partial_t E \ =\ \curl B \ -\ 
\sigma(t,x,x/h)E,
\qquad
\mu(x,x/h)\,\partial_t B  \ = \ - \curl E\,.
\end{equation}
Here $\sigma(t,x,y)$ is a nonnegative symmetric
matrix valued function.  The physics modelled is that
where $\sigma\ne 0$ the medium reacts instantaneously
to the field
$E$ by generating a current $J=\sigma\,E$.
Such an assumption is realistic only when the field
$E$ varies little on the time scales associated to the 
motion of electrons.
The associated energy dissipation law is 
$$
\partial_t \int_{\RR^3}
 \langle
\epsilon E\,,\, E
\rangle
\ +\ 
\langle
\mu B\,,\, B
\rangle
\ dx
\ =\ 
-
\int_{\RR^3}
\big\langle
\sigma E\,,\, E
\big\rangle
\ dx
\ \le \ 0\,.
$$

The lower order tem is $\caO (1)$ and appreciably affects
the fields for times $t=\caO (1)$.  
In this section only, we replace 
Hypothesis \ref{hyp:diffractive} by the following that
allows \eqref{eq:ohm}.  The perturbations are larger by
a factor $h^{-1}$ than in the diffractive case.
\begin{hypothesis}
\label{hyp:gop}
{\bf $T=\Op(1)$ hypothesis.}
\begin{equation}
0\ =\ P^h(t,x,\partial_t,\partial_x) u^h \ :=\ 
\partial_t (A_0^h u^h )\ +\
 \sum_{j=1}^3 A_j \partial_{x_j} u^h 
 \ +\ 
 M^h\, u^h\,.
\end{equation}
The coefficients $A_j$, for $j=1,2,3$, are the constant $6\times 6$ matrices
\begin{equation}
\label{eq:AintermsofJ}
A_1:= \left(
\begin{array}{rr}
0      & J_1  \\
-J_1 & 0
\end{array}
\right)\,,\quad
A_2:= \left(
\begin{array}{rr}
0      & J_2  \\
-J_2 & 0
\end{array}
\right)\,,\quad
A_3:= \left(
\begin{array}{rr}
0      & J_3  \\
-J_3 & 0
\end{array}
\right)\,,
\end{equation}
\begin{equation}
\label{eq:defJ}
J_1:= \left(
\begin{array}{rrr}
0 & 0 & 0 \\
0 & 0 & 1 \\
0 & -1 & 0 
\end{array}
\right)\,, \quad
J_2:= \left(
\begin{array}{rrr}
0 & 0 & -1 \\
0 & 0 & 0 \\
1 & 0 & 0 
\end{array}
\right)\,,\quad
J_3:= \left(
\begin{array}{rrr}
0 & 1 & 0 \\
-1 & 0 & 0 \\
0 & 0 & 0 
\end{array}
\right)\,.
\end{equation}
The coefficient $A_0^h$ and $M^h$ are of the form
\begin{equation}
\label{eq:AzeroandM}
A_0^h(t,x) \ =\
 A_0^0 (x/h) + h A_0^1(t,x,x/h)\,,
 \quad
 M^h=M(t,x,x/h)\,,
\end{equation}
where $A_0^0$ and $A_0^1$ 
satisfy Hypothesis \ref{hyp:diffractive}.
\end{hypothesis}
 
\begin{remark}
\label{rmk:gopgrowth}
It follows that the growth rate $C(m,h)$ from Theorem \ref{thm:stability1} 
(and Remark \ref{rem:diffrsize}) is of order $\caO (1)$ as $h\to 0$.
\end{remark}

Motivated by the special case of purely periodic media
in Section \ref{ppgop}
and the linear case of Lax \cite{lax} (see also
\cite{rauch})
 try 
the {\it ansatz} of two scale WKB type,
\begin{equation}\label{eq:approx-sol}
v^h(t,x)\  :=\  
e^{iS(t,x)/{h}} 
W\Big(h,t,x,\frac{x}{h}\Big)
,
\qquad
W\big(h,t,x,y\big)
\ =\ 
w_0(t,x,y)
+
 h  w_1(t,x,y) \,,
\end{equation}
where the $w_j(t,x,y)$ are periodic
functions of $y$ with period $2\pi$.
The case when $S$ is a linear function of $(t ,x)$ is
our principal interest 
$$
S(t,x)
\ =\
 \omega\,t + \theta\, .x\,,
 \qquad
 (\omega,\theta)\in \RR^{1+3}\setminus 
 \{0\} \,.
$$
For this phase  the rays will be  parallel straight lines and one finds
Schr\"odinger type equations at the diffractive scale  $t=\Op(1/h)$.

Three identities are at the heart of checking the accuracy of the 
{\it ansatz}
\begin{equation}
\label{eq:derivativeformulas}
\begin{aligned}
\partial_t \Big[ A_0^h \, e^{i S(t,x)/h} \, W \big(h,t,x,y \big) \big] 
\ = \  &
e^{iS(t,x)/h} 
\Big(\frac{i\omega}{h}+ \partial_t \Big)\Big[A_0^h \, W\big(h,t,x,y\big)\Big] \,,
\\
\curl_x \Big[e^{i S(t,x)/h}\, W \big(h,t,x,y \big)\Big] 
\ =\ &
e^{iS(t,x)/h}\, \Big(\,\partial_x + \frac{i\theta}{h}\, \Big)\wedge 
W \big(h,t,x,y \big)\,, 
\\
\frac{1}{h}\curl_y \Big[e^{iS(t,x)/h}\, W \big(h,t,x,y\big)\Big] 
\ = \  &
e^{iS(t,x)/h} \,\frac{1}{h} \,\partial_y \wedge W \big(h,t,x,y \big)\,.
\end{aligned}
\end{equation}
These yield
\begin{equation}
\label{eq:ansatz}
 P^h(t,x,\partial_t,\partial_x) v^h \  =\ 
e^{iS(t,x)/h} \,
Z^h(t,x,x/h)\,,
\qquad
{\rm with }
\end{equation}
\begin{equation*}
 Z^h:= 
\left[
\Big(\frac{i\omega}{h}+ \partial_t \Big)A_0^h  
-\left(
\begin{array}{cc}
0  & \big(\,\partial_x + \frac{i\theta}{h}\, \big)\wedge \\
-\Big(\,\partial_x + \frac{i\theta}{h}\, \Big)\wedge  & 0    
\end{array}
\right)
-
\left(
\begin{array}{cc}
0  & \frac{1}{h} \,\partial_y \wedge \\
-\frac{1}{h} \,\partial_y \wedge  & 0    
\end{array}
\right)+M
\right]
W .
\end{equation*}
Then 
\begin{equation}\label{residual}
Z^h(t,x,y) \ =\
 h^{-1} \, r_{-1} \, + \,  r_0 \, + \, h \, r_1\ +\ h^2 r_2 \,,
\qquad
r_j
\ =\
r_j(t,x,y)\,.
\end{equation}
Since one substitutes $y=x/h$,
 it would suffice to satisfy $r_j=0$ on the subspace of
 $(t,x,y)$
with $x$ parallel to $y$.  We
 achieve the more ambitious goal of 
 choosing the $w_j$ so that
$r_{-1}=r_0 =0$ 
everywhere.

\subsection{The leading order term.}
The leading two orders in $Z^h$ are 
$$
h^{-1}r_{-1} \ +\ r_0\ =\ 
h^{-1}\ \LL(\omega,\theta, y,\partial_y) W
\ +\
\Big(\MM(\omega,\theta,y,\partial_t,\partial_x,\partial_y)
\ + \ 
i \omega A_0^1+\ M\Big)W\,,
 $$
where $\LL$ is from \eqref{defL} and 
\begin{equation}
\label{eq:firstM}
\MM(y,\partial_t,\partial_x)\, :=\, 
A_0^0 (y)\partial_t \,
- \left(
\begin{array}{cc}
0 & \partial_x \wedge \\
- \partial_x \wedge & 0
\end{array}
\right) \,.
\end{equation}
The highest order term in \eqref{residual} is 
\begin{equation}
\label{hj0}
r_{-1}
\ =\ 
\LL(\omega,\theta, y,
\partial_y)
\,w_0 \,.
\end{equation}
In order that $r_{-1}=0$ have 
nontrivial solutions,
it is necessary  that
\begin{equation*}
\label{kernel}
{\rm ker}\,\LL(\omega,\theta, y,
\partial_y)\ \neq \ \{0\}\,.
\end{equation*}
According to the Floquet-Bloch theory of 
Section \ref{sec:Bloch-theory}, 
$\LL(\omega,\theta, y, \partial_y)$ has a nontrivial kernel 
on periodic functions if and only $i\omega = i\omega(\theta)$ 
is an eigenvalue of \eqref{eq:evalue}.
From now on we make the choice of $\utheta$ and $\omega(\utheta)$ 
so that the constant multiplicity hypothesis 
\ref{hyp:constantmultiplicity} is satisfied.

\begin{definition}
\label{def:Q}
In addition to Definition  \ref{def:L-P} denote
by
$\QQ\in {\rm Hom}\big(H^{s}(\TT^3_y);H^{s+1}(\TT^3_y)\big)$ 
the partial inverse of 
$\LL$
defined by
$$
\QQ \, \Pi =\Pi\, \QQ =0\,,
\qquad
\QQ\,\LL\ =\ \LL\,\QQ\ =\ I-\Pi\,.
$$
\end{definition}

The equation 
$r_{-1}=0$ is equivalent to 
$w_0\in \KK:= \ker \LL$, that is 
\begin{equation}
\label{polarization}
\Pi\, w_0 = w_0\,.
\end{equation}
Using the definitions of $\LL$ and $\MM$, the
 term $r_{0}$ is given by,
\begin{equation}
\label{eq:r0}
r_{0} 
\ =\
 \LL w_1 \ +\
  \big(\MM +  \ i \omega A_0^1+ M\big) w_{0},
\end{equation}
so
 $r_{0}=0$ if and only if,
\begin{equation}\label{eq:r0-vanish}
\LL\, w_1\
+  \big(\MM + \ i \omega A_0^1+ M\big)w_{0}
\ = \
0\,.
\end{equation}
Equation \eqref{eq:r0-vanish} involves both $w_0$ and $w_1$.
Since $\LL$ is selfadjoint on $L^2(\TT^3)$ its range is perpendicular
to its kernel so $\Pi\, \LL =0$.
This is true only because $\Pi$ is the orthogonal projection 
on the kernel $\KK$ for the $L^2(\TT^3)$ scalar product 
(not for the other scalar product in Definition \ref{def:L-P}). 
The equation $\Pi\,r_{0}=0$ yields an equation
for $w_0$ alone,
\begin{equation*}
\Pi\,
\big(\MM \ + \ i \omega A_0^1+\ M\big)
w_0\ =\ 0\,.
\end{equation*}
Taking into account \eqref{polarization},
 this is 
equivalent to 
\begin{equation}
\label{transpt}
\Pi\big(
\MM\ + \ i \omega A_0^1+ \ M\big)\Pi\, w_0\ =\ 0\,.
\end{equation}

\begin{proposition}  
\label{prop:groupvel}
For any $w(t,x,y)\in C^\infty$,
\begin{equation}
\label{eq:Msandwich}
\Pi\, \MM\,\Pi \, w
\ =\ 
\Pi \, A_0^0 \, \Pi \,
\big(\partial_t + \V .\partial_x \big)\, w \,,
\end{equation}
with the group velocity $\V$ from Definition \ref{def:groupvelocity}.
The operator $\Pi \, A_0^0 \, \Pi$ is a linear isomorphism of $\KK$ to itself.
\end{proposition}

\begin{proof}  Prove the last sentence first.   Since $\KK$ is finite
dimensional, it suffices to prove injectivity.  If $\Pi\, A_0\,\Pi\, k=0$
then
$$
0 \ =\
\big\langle
\Pi\, A_0\,\Pi\, k
\,,\,
k
\big\rangle
\ =\
\big\langle
A_0\,\Pi\, k
\,,\,
\Pi k
\big\rangle
\ \ge\ 
c\,\| \Pi k\|^2
\,,
\qquad c>0\,,
$$
since $A_0$ is strictly postive.  Therefore $k=\Pi k =0$
proving injectivity.

From the definition of $\Pi$ and $\MM$ one automatically
has for arbitrary $w$,
\begin{equation*}
\Pi\, \MM\,\Pi \, w
\ =\ 
\Big(
a_0\,\partial_{t}
\ +\ 
\sum_{j=1}^3
 \,a_j\,
\, \frac{\partial }{\partial x_j}
\Big)\Pi\,w\,,
\end{equation*}
with matrices $a_j(y)$.  
It suffices to compute the $a_j$.
This is done by computing the differential operator
on the  test functions $t \, \psi$, and
$x_j\psi$, with $\psi\in\KK$.
Applying to $t\psi$ and setting $t=0$ yields
\begin{equation}
\label{eq:azero}
a_0\,\Pi 
\ =\ 
\Pi \, A_0^0 \, \Pi \,
\,.
\end{equation}
Applying to $x_j\psi$ and setting $x_j=0$ yields
\begin{equation}\label{eq:aj}
a_j\,\Pi
\ =\ 
-\,
\Pi
\left(
\begin{array}{cc}
0 & e_j\, \wedge \\
-e_j \,\wedge & 0
\end{array}
\right)\Pi \,,
\qquad
\{e_1,e_2,e_3\}
\textup{ is the standard basis of 
}
\RR^3\,.
\end{equation}
The 
identification of $a_j$ from \eqref{eq:aj} and 
\eqref{eq:azero}
requires first order perturbation theory as in 
\eqref{perturb}
of the next proposition.  Second order perturbation theory as
in \eqref{perturb2}
is needed 
for diffractive geometric optics.  The identites are proved by
differentiating the identities $\Pi \,L =0$ and $L\, \Pi=0$. 
We refer the reader to \cite{apr2}, \cite{rauch}
 for  detailed proof.

\begin{proposition}
\label{perturbationtheory}
Suppose that $\utheta$ and 
$\omega$ satisfy the constant multiplicity hypothesis 
\ref{hyp:constantmultiplicity}.   
Suppose that the coefficient $A_0^0$ 
and $\theta$ depend
smoothly on a parameter $\alpha$ with their unperturbed
values attained at $\underline\alpha$.
With ${}^\prime$ denoting differentiation with respect to 
$\alpha$, the following perturbation formulas hold,
\begin{equation}
\label{perturb}
\Pi\,\LL^\prime\,\Pi \ =\ 0\,,
\end{equation}
and
\begin{equation}
\label{perturb2}
\Pi\,\LL^{\prime\prime}\,\Pi 
\ -\ 
2\,\Pi\,\LL^\prime\,Q\,\LL^\prime\,\Pi
\ =\ 0\,.
\end{equation}
\end{proposition}

Returning to the formula for $a_j$,
use \eqref{perturb} with $\alpha$ equal
to the $j^{\rm th}$ component of $\theta$ and
 ${}^\prime =\partial/\partial {\theta_j}$.
Then,
\begin{equation*}
\LL^\prime \ = \ 
i\,\frac{\partial \omega}{\partial\theta_j}\, A_0^0(y)
\  -\  
\left(
\begin{array}{cc}
0 & i \, e_j \wedge \\
-i\, e_j \wedge & 0
\end{array}
\right)\,.
\end{equation*}
The above identity in combination with \eqref{perturb} and \eqref{eq:aj} yeilds 
\begin{equation*}
a_j \,\Pi \ =\
-\,
 \Pi \, A_0^0(y)\, \Pi  \,\frac{\partial \omega}{\partial\theta_j}\,.
\end{equation*}
This together with \eqref{eq:azero}
completes the proof of Proposition \ref{prop:groupvel}.
\end{proof}

Recall Definition \ref{def:gamma}.
The 
map $\gamma$ inherits the regularity of the coefficients,
\begin{equation}
\label{eq:derivgamma}
\partial_{t,x}^\alpha\gamma\in L^\infty\big(\RR^{1+3}
\, ;\,
{\rm Hom}\,\KK\big)
\,.
\end{equation}
Then $r_{-1}=\Pi r_0=0$ exactly when $w_0=\Pi w_0$
satisfies the transport equation
\begin{equation}
\label{eq:w0transport}
\Big(
\partial_t \ +\
\V.\partial_x\ +\
\gamma(t,x)\Big)
w_0\ = \ 0\,.
\end{equation}
The equation $r_0 = 0$ is equivalent to the pair
$$
\Pi \, r_0 \ =\ 0\,, \qquad \Q \, r_0 \ =\  0 \,.
$$
Equation \eqref{eq:r0} shows that $\Q\, r_0 = 0$ if and only if 
\begin{equation}
\label{eq:Qw1}
\big(I-\Pi\big)w_1 \ =\  - \Q \big(\MM + \ i \omega A_0^1+ M \big)w_0 \,.
\end{equation}
The choice of $\Pi \, w_1$ does not influence $r_1, r_0$. Choose 
\begin{equation}\label{eq:Piw1}
\Pi\, w_1 = 0 \,.
\end{equation}

\begin{theorem}\label{error-est1}
If $g\in C^\infty(\R^3;\KK)$ there is a unique
$w_0\in C^\infty(\R;H^\infty(\R^3;\KK))$ satisfying \eqref{eq:w0transport} with
$w_0(0)=g$.
Define $w_1$ and $v^h$ by \eqref{eq:Qw1}, \eqref{eq:Piw1} and \eqref{eq:approx-sol} respectively.
If $u^h$ is the exact solution of $P^h u^h = 0$
with 
$u^h\big|_{t=0}=v^h\big|_{t=0}$, then for all $\alpha$
$$
\sup_{t\in [0,T]}\ \
\big\|
(h\,\partial_{t,x}  )^\alpha
(u^h - v^h)
\big\|_{L^2(\R^3)} \ \leq\  C(\alpha)\,h \,,
 \qquad 0<h<1\,.
$$
\end{theorem}

\begin{proof}
{\bf I.  Estimate for $P^hv^h$.}    Use \eqref{eq:ansatz}
together with \eqref{residual} and the fact that the
equations satisified by the  functions
$w_j$ guarantee that $r_{-1}=r_0=0$.
Therefore for $h\in ]0,1[$
$$
\big\|
\partial_{t,x,y}^\alpha Z^h
\big\|_{L^\infty([0,T]\times \RR^{3}\times \TT^3)}
\ \le \
C(\alpha)\, h,
\quad
\text{ supp}\, Z^h \subset \big\{(t,x+\V t,y)\, : \, x\in\text{ supp }g  \big\}
\,.
$$
This implies the fundamental residual estimate
\begin{equation*}
\big\|
(h\,\partial_{t,x}  )^\alpha
 \big(P^h v^h\big)
\big\|_{L^\infty([0,T]\times \RR^{3})}
\ \le \
C(\alpha)\, h,
\qquad
h\in ]0,1[\,.
\end{equation*}
Using the compact support one has 
\begin{equation}
\label{eq:residual}
\big\|
(h\,\partial_{t,x}  )^\alpha
 \big(P^h v^h\big)
\big\|_{L^\infty([0,T]; L^2(\RR^{3}))}
\ \le \
C(\alpha)\, h,
\qquad
h\in ]0,1[\,.
\end{equation}

{\bf II.  Stability for $P^h$.}   For $T>0$ and $m$ fixed 
Theorem \ref{thm:stability1} shows that there is a constant $C=C(T,m)$
so that for $0\le t\le T$ and $h\in ]0,1]$
\begin{equation}
\label{eq:stability1}
\big\|
w(t)\big\|_{m,h}
\ \le\
C\,\Big(
\big\|
w(0)\big\|_{m,h}
\ +\
\int_0^t
\big\|
P^h\,w(s)\big\|_{m,h}\ ds
\Big)\,.
\end{equation}

{\bf III.  Combining.}
Apply \eqref{eq:stability1} to $w^h:=u^h-v^h$ to find
\begin{equation*}
\label{eq:difference}
\begin{aligned}
\big\|
(u^h-v^h)(t)\big\|_{m,h}
&\ \le\
C\,\Big(
\big\|
(u^h-v^h)(0)\big\|_{m,h}
\ +\
\int_0^t
\big\|
P^h\,(u^h-v^h)(s)\big\|_{m,h}\ ds
\Big)
\cr
&\ =\
C\,\Big(
\big\|
(u^h-v^h)(0)\big\|_{m,h}
\ +\
\int_0^t
\big\|
P^h\,v^h(s)\big\|_{m,h}\ ds
\Big)
\cr
&\ =\
C\,\Big(
\big\|
(u^h-v^h)(0)\big\|_{m,h}
\ +\
\Op(h)\Big)
\,.
\end{aligned}
\end{equation*}
where the last step uses
the residual estimate \eqref{eq:residual}.

It remains to show that
$w^h=u^h-v^h$ satisfies for
for all $m$,
\begin{equation}
\label{eq:initialestimate2}
\| w^h(0)\|_{m,h} \ =\ \Op(h),
\qquad
h\to 0\,.
\end{equation}
To do that it suffices to show that
for all $0\le k\in \NN$ and $s$,
\begin{equation}
\label{eq:initialestimate}
\|(h\partial_t)^k w^h(0) \|_{H^s_h(\RR^3)}\ =\ \Op(h),
\qquad
h\to 0\,.
\end{equation}
 Since the initial
values vanish, the case $k=0$ is automatic.
Prove \eqref{eq:initialestimate}   by induction on $k$.  Suppose
\eqref{eq:initialestimate} 
known for $0\le k\le \underline{k}$ and all $s$.  We prove it for $\underline{k}+1$
and all $s$.
Begin with the identity
$$
h (A_0^h)^{-1} P^h w^h \ = \
h (A_0^h)^{-1} \big(A_0^h\partial_t w^h + (\partial_t A_0^h)w^h + 
\sum_j A_j \partial_j w^h + M^h w^h 
\big).
$$
The first term on the right is $h\partial_tw^h$.  Therefore
$$
(h\partial_t)^{\underline k+1}w^h
=
(h\partial_t)^k
h\partial_tw^h
=
(h\partial_t)^{ \underline k}
(A_0^h)^{-1}\,
h\,\Big(
P^h
-
(\partial_t A_0^h) -
\sum_j A_j \partial_j -  M^h 
\Big)
w^h.
$$
The inductive hypothesis
 implies that
 $\|(h\,\partial_t)^{\underline{k }+1} u(0)\|_{H^s_h(\RR^3)}
 = \Op(h)$.  This completes the inductive proof
 of \eqref{eq:initialestimate} and therefore the proof of the theorem.
\end{proof}


\section{Bloch wave packets on a modulated background and $t=\Op(1/h)$}
\label{sec:gop}

This section considers solutions of the Maxwell's equations
with coefficients and lower order terms satisfying
Hypothesis \ref{hyp:diffractive}  so as to 
become pertinent at $t\sim 1/h$.

\begin{remark}
\label{rmk:diffr}
With permittivities satisfying
Hypothesis \ref{hyp:diffractive},
the growth rate $C(m,h)$ from Theorem \ref{thm:stability1}
satisfies $C(m,h)\le c(m)\,h$ 
with a constant $c(m)$ independent of $h\in ]0,h[$.
The time evolution is uniformly bounded so long as 
the product $\,t\times h$ 
remains bounded.
\end{remark}

Motivated by the special case of purely periodic media
in Section \ref{ppgop},
the {\it ansatz} expected to be valid for 
times $t=\Op(1/h)$  is of two scale WKB type,
\begin{equation}
\label{eq:approx-diffr}
v^h(t,x)\ :=\ 
e^{iS(t,x)/{h}} 
W\Big(h,ht, t,x,\frac{x}{h}\Big)
\,,
\qquad
S(t,x)
\ =\
 \omega\,t + \theta\, .x\,,
\end{equation}
$$
W\big(h,\caT, t,x,y\big)
\ =\ 
w_0(\caT, t,x,y)
\ +\ 
 h  w_1(\caT , t,x,y) 
  \ +\ 
 h^2  w_2(\caT, t,x,y) 
,
$$
where the $w_j(\caT, t,x,y)$ are periodic
functions of $y$ with period $2\pi$.
To go further in time requires 
the additional corrector 
$w_2$.
In order to preserve the relative ordering of the 
size of the terms for $t=\Op(1/h)$
 it is crucial to insist that {\it the 
$w_j(\caT,t,x,y)$ are sublinear in $t$,}
$$
\lim_{t\to \infty}
\ \ 
 \frac{w_j(\caT,t,x,y)}{t}\ =\ 0
$$
uniformly in $\caT,x,y$.    We construct
profiles satisfying a stronger hypothesis.

Compute using \eqref{eq:derivativeformulas}
\begin{equation}
\label{eq:ansatz2}
P^h(t,x,\partial_t,\partial_x) v^h \  =\ 
e^{iS(t,x)/h} \,
Z^h(h\caT, t,x,x/h)\,,
\qquad {\rm with}
\end{equation}
\begin{equation}
\begin{aligned}
 Z^h(\caT,   &  t,x,y)\ := \
 \Bigg[
\Big(\frac{i\omega}{h}+ \partial_t
+h\partial_\caT \Big)A_0^h  
\ -\ 
\\
&\begin{pmatrix}
0  & \big(\,\partial_x + i\theta/h\, \big)\wedge 
\\
\ -\
\big(\,\partial_x + i\theta/h\, \big)\wedge  & 0    
\end{pmatrix}
-
\begin{pmatrix}
0  & \frac{1}{h} \,\partial_y \wedge \\
-\frac{1}{h} \,\partial_y \wedge  & 0    
\end{pmatrix}
+hM
\Bigg]
W \,.
\end{aligned}
\end{equation}
Then
\begin{equation}
\label{residual2}
Z^h(t,x,y) \ =\
 h^{-1} \, r_{-1} \, + \,  r_0 \, + \, h \, r_1\ +\ h^2\,r_2
 \ +\ h^3\,r_3 \,,
\qquad
r_j
\ =\
r_j(t,x,y)\,.
\end{equation}
As in the case of $t=\Op(1)$, we
 achieve the ambitious goal of 
 choosing the $w_j$ so that
 the leading
$r_{-1}=r_{0} = r_1=0$ 
everywhere.  This reduces the residual to $\Op(h^2)$ which
allows us to justify for times $t=\Op(1/h)$.

One has
$$
r_{-1} 
\ =\
 \LL  w_0 \qquad \text{and} \qquad
r_0
\ =\
 \LL w_1 + \MM w_0
$$
with $\LL$ and $\MM$ defined by \eqref{defL} and \eqref{eq:firstM} respectively.
The next coefficient is
\begin{equation*}
\begin{aligned}
r_1 \, = \, \,  & \partial_t  A_0^0 w_1 + \partial_\tau(A_0^0 w_0) + 
i\omega(A_0^0 w_2 + A_0^1 w_0) -
\begin{pmatrix}
0  & \partial_x \wedge 
\\
\ -\
\partial_x \wedge  & 0    
\end{pmatrix} w_1 \\
&\hskip3.5cm  - 
\begin{pmatrix}
0  & \big(\, i\theta+ \partial_y \, \big)\wedge 
\\
\ -\
\big(\, i\theta+ \partial_y \, \big)\wedge  & 0    
\end{pmatrix} w_2  
+ M w_0 \\
= \,\, & 
 \LL w_2 + \MM w_1
 + \NN w_0 \,,
\end{aligned}
\end{equation*}
where
\begin{equation}\label{defN}
\NN
\ :=\
 \partial_\caT A_0^0   + i\omega A_0^1 + M  \,.
\end{equation}
As in Section \ref{sec:gop},   $r_{-1}=0$ and $\Pi\ r_0 = 0$ 
 if and only if
$$
\Pi w_0 = w_0\,, \qquad
{\rm and}
\qquad
 (\partial_t + \V .\partial_x) w_0 = 0 \,,
$$
with $\V$ as in Definition \ref{def:groupvelocity}.
Thus there is a reduced $\KK$ valued profile $\widetilde{w}_0$ 
such that 
$$
w_0(\caT,t,x,y)
\ =\
 \widetilde{w}_0(\caT ,x-\V t,y) \,.
$$
In order to determine $\widetilde{w}_0$ one needs a 
dynamic equation in $\caT$.
The equation $\Q \ r_0 = 0$ yields 
\begin{equation}\label{eq:(I-Pi)w1}
\big(I-\Pi\big)w_1 \ =\  - \Q \ \MM \ w_0 \,,
\end{equation}
and thus 
\begin{equation}\label{eq:transportw1}
(\partial_t + \V. \partial_x)\big(I-\Pi\big)w_1 \ =\ 0 \,.
\end{equation}
The equation $\Pi r_1=0$ yields the Schr\"odinger equation determining the 
dynamics of $\widetilde{w}_0$. 
Plugging \eqref{eq:(I-Pi)w1} into the equation $\Pi r_1=0$ yields 
\begin{equation}\label{eq:r1=zero}
\Pi\ \MM\ \Pi w_1 -\Pi\ \MM\ \Q\ \MM\ \Pi w_0 
+\Pi\ \N\ \Pi w_0 = 0 \,. 
\end{equation}
Proposition \ref{prop:groupvel} and \eqref{eq:transportw1} imply that 
equation \eqref{eq:r1=zero} is 
equivalent to
\begin{equation}\label{eq:rough-schrod}
\big(
\Pi\ \N\ \Pi - \Pi\ \MM\ \Q\ \MM\ \Pi 
\big)w_0 \ = \
-\Pi\ A_0^0\ \Pi
\big(\partial_t + \V.\partial_x \big)\Pi w_1  \,.
\end{equation}
The next proposition  identifies the operator on the 
left hand side of \eqref{eq:rough-schrod}.

\begin{proposition}
\label{prop-schrod}
On smooth functions $w(\caT,t,x,y)$ that satisfy
$(\partial_t + \V. \partial_x)w=0$, 
\begin{equation}
\label{miracle}
\big(\Pi\,\NN\,\Pi 
  \ -\ \Pi\,\MM\,Q\,\MM\,\Pi \Big)\,w
  \ =\ 
\Pi A_0^0\Pi
\Big[
\partial_\caT    +
\frac{1}{2}i \  
\partial^2_\theta\omega(\partial_x,\partial_x)  
+ \gamma(t,x) 
\Big] w \,.
\end{equation}
\end{proposition}

\begin{remark}  
\label{rmk:scalar}
It is surprising to find that the operator 
$$
(\Pi A_0^0\Pi)^{-1}
\big(\Pi\,\NN\,\Pi 
  \ -\ \Pi\,\MM\,Q\,\MM\,\Pi \Big)
  $$
  acting on $\KK$ valued functions  
  has leading terms that are scalar.
  Coupling only occurs through the zero order
  term $\gamma$.
  A related zero order coupling occurs in 
  \S 6 of 
  \cite{ap}
  \end{remark}  

\begin{proof}
With $k=(k_1,k_2,k_3)\in\R^3$ fixed,
$\theta:=\utheta + \alpha k$, $\alpha\in\R$, differentiate 
$\LL(\omega,\theta,y,k)$ with respect to $\alpha$ to find
\begin{equation}\label{eq:alpha-der1}
\begin{aligned}
& \LL'= ik A_0^0\partial_\theta\omega - 
\begin{pmatrix}
0 & i k \wedge \\
-ik \wedge & 0
\end{pmatrix} =
i\MM(y,\partial_t,k) - iA_0^0\,\DD(\partial_t ,k)\,,\\
& \LL''= i (k\partial_\theta)^2\omega A_0^0 \,.
\end{aligned}
\end{equation}
Using the first identity yields
\begin{equation*}
\Pi\ \LL'\ \Q\ \LL'\ \Pi \  =\
 - \Pi \ \MM\ \Q\ \MM\ \Pi  -
\Pi\ A_0^0\ \DD(\partial_t ,k) \ 
\Q\ A_0^0\ \DD(\partial_t ,k) \ \Pi \,.
\end{equation*}
Applying \eqref{perturb2} yields
\begin{equation}
\label{eq:alpha-der2}
 \frac{1}{2}\, \Pi\ \LL'' \ \Pi 
  \ =\ 
 - \Pi \ \MM\ \Q\ \MM\ \Pi  -
\Pi\ A_0^0\ \DD(\partial_t ,k) \ 
\Q\ A_0^0\ \DD(\partial_t ,k) \ \Pi \,.
\end{equation}
Equation \eqref{eq:alpha-der2} and  the 
definition of $\N$ in \eqref{defN}, give
\begin{equation}\label{almost-schrod}
\begin{aligned}
\Big(\Pi\,\NN\,\Pi   &
  \ -\     \Pi\,\MM\,Q\,\MM\,\Pi \Big)\,w
  \\
  \ =\ &
\Pi\Big[
\partial_\caT A_0^0  + i\omega A_0^1 + M +
\frac{1}{2}i \  (k\partial_\theta)^2\omega A_0^0  
  + A_0^0\,\DD(\partial_t ,k) \ 
\Q\ A_0^0\,\DD(\partial_t ,k)
\Big]\Pi w \,.
\end{aligned}
\end{equation}
Next replace $k$ by $\partial_x$ and simplify the right hand side of 
\eqref{almost-schrod} using 
$$
\DD(\partial_t,\partial_x)w=0  \,.
$$
This yields
\begin{equation*}
\begin{aligned}
\Big(\Pi\,\NN\,\Pi  
  \ -\     \Pi\,\MM\,Q\,\MM\,\Pi \Big)\,w
  &\ =\ 
\Pi\Big[
\partial_\caT A_0^0  + 
i\omega A_0^1  
  + M +
\frac{1}{2}i \  
\partial^2_\theta\omega(\partial_x,\partial_x) A_0^0    
\Big]\Pi w  \\
&\ = \
\Pi A_0^0\Pi
\Big[
\partial_\caT    +
\frac{1}{2}i \  
\partial^2_\theta\omega(\partial_x,\partial_x)  
+ \gamma(t,x) 
\Big] w \,.
\end{aligned}
\end{equation*}
\end{proof}

\subsection{Ray averages.}
\label{sec:rayaverages}

In general it is impossible to satisfy \eqref{eq:rough-schrod} exactly
since all the terms are annihilated by $\partial_t+\V\partial_x$
except the $\gamma(t,x)$ term.  If the coefficients are constant
on group lines the $\gamma$ term is constant too and one can construct
infinitely accurate solutions (see \cite{apr2}).  In the present case
we replace $\gamma$ by  its average on rays to find a solvable
equation.  Then estimate the error.  For that estimate we
impose an assumption on $\gamma$ slightly stronger
than the existence of ray averages.  The material is recalled
from \cite{apr2} where the proofs and additional
information can be found.

Assume that $\gamma\in C^\infty(\RR^{1+3}\,;\,{\rm Hom}(\KK))$
satisfies \eqref{eq:derivgamma} and that the averages on rays
exist as in \eqref{eq:meangamma}.
It follows that the $\widetilde \gamma$ is smooth and that
the ray averages of the derivatives of $\gamma$ exist uniformly
on compacts and satisfy
\begin{equation}
\label{eq:derivmean} 
\lim_{T\to +\infty}\
\Big\|
\frac{1}{T}\,
\int_0^T
\partial_t^j\partial_x^\beta
\gamma(t,x+\V t)
\ dt
\ -\
(-\V.\partial_x)^j\partial_x^\beta\widetilde\gamma(x)
\Big\|_{L^\infty(\RR^N)}
\ =\ 0\,.
\end{equation}

 We need  more than this.  
The function $\widetilde \gamma(x)$ is the average on the
ray intersecting $t=0$ at $x$. 
The ray passing through the point $(t,x)$
intersects $t=0$ at $x-\V t$.   The function which 
assigns to $(t,x)$ the average value of $\gamma$ on the 
ray through $(t,x)$ is equal to $\widetilde\gamma(x-\V t)$. 
The function
that subtracts from $\gamma(t,x)$ its 
average on the group line through $(t,x)$ is 
equal to
$\gamma(t,x)-\widetilde\gamma(x-\V t)$.

Consider the ${\rm Hom } (\KK)$ valued solution $g$ of the 
transport equation 
\begin{equation}
\label{eq:2}
\Big(
\partial_t \ +\ 
\V.\partial_x\Big)g
\ =\ 
\gamma(t,x) -\widetilde\gamma(x-\V t)
\,,
\qquad
g\big|_{t=0} = 0\,.
\end{equation}
Then
\begin{equation*}
\frac{g(t,x)}{t}
\ =\ 
\frac{1}{t}
\int_0^t \gamma(s, \widetilde x + \V s)\,ds
\ -\ 
\widetilde \gamma(\widetilde x),
\qquad
\widetilde x := x-\V t
\,.
\end{equation*}
Assumption \eqref{eq:meangamma}
is equivalent to the fact 
that this
is $o(1)$ as $t\to+\infty$,
\begin{equation}
\label{eq:littleo}
\lim_{t\to +\infty}\ \ 
\sup_{x\in \RR^N}\ \ 
\frac{  \| g(t,x) \| }{t}\ =\ 0\,.
\end{equation}

Assume that $\gamma$ satisfies the ray average hypothesis
in Definition \ref{def:rayaverage}.

\begin{example}
{\bf i.}  If $\gamma(t,x)=f(\ell(t,x))$ where $f(\theta)$ is a smooth
periodic function of arbitrary period and $\ell$ is 
a linear functional then the ray average hypothesis is satisfied
with $\beta=0$.

{\bf ii.}  If
${\mathcal M}:\RR^{1+N}\to \RR^M$ is linear and satisfies the 
(generic) small divisor hypothesis
$$
\exists \,C>0, \  m\in \NN,\quad
\forall n\in \NN^{M},\quad
n.{\mathcal M}(1,\V)\ne  0 \ 
\Rightarrow\
|(n.{\mathcal M}(1,\V )| \ \ge \ C\,|n|^{-m}\,,
$$
then, for 
$h(\theta_1,\dots , \theta_M)\in C^\infty(\TT^M)$ 
the quasiperiodic function
$\gamma(t,x) =h({\mathcal M}(t,x))$ satisfies the hypothesis
with $\beta=0$ 
(see \cite{jmrduke}).

{\bf iii.} 
Consider smooth almost periodic $\gamma$ 
of the form
\begin{equation}
\label{eq:almostper}
\gamma(t,x)
\ =\
\sum_{\eta\in \R^{1+N}}
a_\eta\
e^{i\eta.(t,x)}\,,
\end{equation}
where $a_\eta$ vanish for all
but a countable family of $\eta$ and  satisfy
\begin{equation}
\forall n\in \N,\qquad
\label{eq:coeffhyp}
\sum_\eta
\
\langle \eta\rangle^n
\
\big|
a_\eta
\big|
\ <\
\infty
\qquad
\langle
\eta
\rangle:=
(1+|\eta|^2)^{1/2}\,.
\end{equation}
Then 
$
\gamma(t,x) \ -\
\widetilde \gamma(x-\V t)\ =\
\sum_{\eta.(1,\V)\ne 0}
a_\eta\
e^{i\eta.(t,x)}\,.
$
Suppose that   there is an $\alpha>0$ so that
for all $n$
\begin{equation}
\label{eq:lowfreq}
 \sum_{0<|\eta.(1,\V)|<\delta}
\langle \eta\rangle^n\
\big|
a_\eta
\big|
\ =\
\caO(\delta^\alpha),
\qquad
\delta\to 0\,.
\end{equation}
Then the ray average hypothesis
of Definition \ref{def:rayaverage} holds with 
$\beta={\alpha}/(\alpha+1)$.
\end{example}

\subsection{Using ray averages}
Rewrite \eqref{eq:rough-schrod} as
\begin{equation}
\label{eq:9.20}
\Big(\partial_\caT   + 
\frac{1}{2}i \  
\partial^2_\theta\omega(\partial_x,\partial_x)  
+\widetilde \gamma(x-\V t) 
\Big) w_0 
\ = \
-\big(\partial_t + \V.\partial_x \big)\Pi w_1
- \big(\gamma(x) - \widetilde\gamma(x-\V t)\big)w_0
\,.
\end{equation}
This equation is satisfied by choosing $w_0$ and $\Pi w_1$ so that
both sides are identically zero.   The left yields
equation \eqref{eq:transp-schrod}.
The initial value, $\widetilde w_0(0,x)\in {\caS}(\RR^3\,;\,\KK)$
is arbitrary.  

\begin{lemma}
\label{lem:schrodweight}
  For any $f\in \caS(\RR^3;\KK)$ there is a 
unique solution $\zeta(\caT,x)\in C^\infty(\RR_\caT\,;\, \caS(\RR^3;\KK))$ 
to the Schr\"odinger initial value
problem
\begin{equation}\label{operatorS}
\Big(\partial_\caT   + 
\frac{1}{2}i \  
\partial^2_\theta\omega(\partial_x,\partial_x)  
+\widetilde \gamma(x) 
\Big)\zeta
\ = \
0\,,
\qquad
\zeta(0)\ =\ f\,.
\end{equation}
For each $m,s,r\in\NN$, 
$|\alpha|\leq m$, $|\kappa|\leq s$ there exist constants 
$c(m,s,r),b(m,s)$ 
so that
for all $\caT$
$$
\|x^\kappa\partial_{x}^{\alpha} \partial_\caT^r \zeta (\caT)\|_{L^2(\R^3)}
\leq
c(m,s,r)
e^{\big(1+ b(m,s)
\| \widetilde\gamma\|_{_{W^{m+s,\infty}}}\big)
\caT }\ 
(1+\|\widetilde\gamma\|_{W^{r,\infty}})^r
\sum_{|\ell|=0}^s
\|x^\ell f\|_{H^{m+ s -\ell+ 2r}}  \,.
$$
\end{lemma}

\begin{proof}  
We present only the a priori estimates.
Multiplying the equation by $\bar\zeta$ and taking the real part 
yields
$$
\frac{1}{2}\frac{\partial}{\partial\caT}\int_{\R^3}|\zeta|^2 \ dx +
\Re
\int_{\R^3}\langle \widetilde \gamma \zeta , \zeta\rangle \ dx \ = \ 0 \,.
$$
Estimate 
$$
\Re
\int_{\R^3}\langle \widetilde \gamma \zeta , \zeta\rangle  \ dx \leq
\| \widetilde\gamma\|_{_{L^\infty}} 
\int_{\R^3}|\zeta|^2 \ dx \,,
$$
and integrate with respect to $\caT$ to find
$$
\frac{1}{2}\int_{\R^3}|\zeta(\caT)|^2 \ dx 
\leq
\frac{1}{2}\int_{\R^3}|f|^2 \ dx \ + \ 
\| \widetilde\gamma\|_{_{L^\infty}}
\int_0^\caT\int_{\R^3}|\zeta(\caT)|^2 \ dx  \,.
$$
Gronwall's inequality gives
\begin{equation}\label{eq:firstpart}
\|\zeta(\caT)\|_{L^2(\R^3)} \  \leq  \ 
e^{\| \widetilde\gamma\|_{_{L^\infty}}\caT }\ 
\|f\|_{L^2(\R^3)} \,.
\end{equation}
The Duhamel relation 
$$
\begin{aligned}
\partial_{x_j} \zeta (\caT) \ & =\ K(\caT,0) \partial_{x_j} \zeta (0) 
\ +\ 
\int_0^\caT K(\caT,s)\ P(\partial_{x_j} \zeta)(s)\ ds \\
& = 
\ K(\caT,0) \partial_{x_j} f
\ +\ 
\int_0^\caT K(\caT,s)\ (-\partial_{x_j }\widetilde\gamma)\zeta(s)\ ds 
\end{aligned}
$$
then yields the estimate 
$$
\| \partial_{x_j} \zeta (\caT) \|_{L^2(\R^3)}\  \leq \    
e^{\| \widetilde\gamma\|_{_{L^\infty}}\caT }\ 
\| \partial_{x_j} f \|_{L^2(\R^3)} \ + \ 
\int_0^\caT
e^{\| \widetilde\gamma\|_{_{L^\infty}}(\caT-s)} 
\|\nabla \widetilde\gamma\|_{L^\infty} 
\|\zeta(s)\|_{L^2(\R^3)} \, ds \,.
$$
This combined with \eqref{eq:firstpart} gives
$$
\|\zeta(\caT)\|_{H^1(\R^3)} 
\  \leq  \ 
2\ e^{\| \widetilde\gamma\|_{_{L^\infty}}\caT }\ 
\|f\|_{H^1(\R^3)} \ + \ 
\int_0^\caT
e^{\| \widetilde\gamma\|_{_{L^\infty}}(\caT-s)} 
\|\nabla \widetilde\gamma\|_{L^\infty} 
\|\zeta(s)\|_{L^2(\R^3)} \, ds \,.
$$
Now apply Gronwall's inequality to get
$$
\|\zeta(\caT)\|_{H^1(\R^3)} \leq 
2
e^{\| \widetilde\gamma\|_{_{W^{1,\infty}}}\caT }\ 
\|f\|_{H^1(\R^3)} \,.
$$
By induction one proves that 
$$
\sup_{|\alpha|\le m} 
\big\|
(\partial_x)^\alpha\, \zeta(\caT)
\big\|_{L^2(\RR^3)}
\  \le \
2 \ 
e^{b(m) \, \| \widetilde\gamma\|_{_{W^{m,\infty}}}\caT }\ 
\| f \|_{H^m(\R^3)} \,.
$$
Let us prove the weighted estimate  $x^\kappa\partial_x^\alpha$.
The commutator of $x_j$ with the Schr\"odinger operator $S$ 
in \eqref{operatorS} is the 
first order scalar differential operator
$$
[S, x_j] = 
i \sum_l (\partial_\theta^2\omega)_{lj}\partial_{x_j}\,.
$$
Therefore, for $|\alpha|\leq m$ 
$$
\begin{aligned}
x_j\partial_{x}^{\alpha} \zeta (\caT) \ & =\ 
K(\caT,0) x_j\partial_{x}^{\alpha} \zeta (0)
\ +\ 
\int_0^t K(\caT,s)\ S(x_j\partial_{x}^{\alpha} \zeta)(s)\ ds \\
& = 
\ K(\caT,0) x_j\partial_{x}^{\alpha} f
\ +\ 
\int_0^t K(\caT,s)\ 
[S, x_j] (\partial_{x}^{\alpha} \zeta)(s)
\ ds  \,,
\end{aligned}
$$
which yields
$$
\|x_j\partial_{x}^{\alpha} \zeta (\caT)\|_{L^2(\R^3)}
\leq
c \,
e^{\big(1 + b(m+1)\| \widetilde\gamma\|_{_{W^{m+1,\infty}}}\big)
\caT }\ 
\Big[
\|x_j g\|_{H^m(\R^3)} \ + \ 
\| f \|_{H^{m+1}(\R^3)} 
\Big]
\,.
$$
By induction one proves, for $|\kappa| \leq s$,
\begin{equation}\label{eq:ugly}
\|x^\kappa\partial_{x}^{\alpha} \zeta (\caT)\|_{L^2(\R^3)}
\leq
c(m,s)
e^{\big(1 + b(m,s)
\| \widetilde\gamma\|_{_{W^{m+s,\infty}}}\big)
\caT }\ 
\sum_{|\ell|=0}^s
\|x^\ell  f \|_{H^{m+s-\ell}}  \,.
\end{equation}
The time derivative commutes with the Schr\"odinger operator, 
therefore, for each $r\in \NN$,  
$\partial_\caT^r \zeta$ satisfies the same equation as $\zeta$ 
with initial condition 
$$
|\partial_\caT^r\zeta(0)| \leq 
C(r)(1+\|\widetilde\gamma\|_{W^{r,\infty}})^r
\sum_{j\leq 2r}| \partial^j_\caT f | \,.
$$
Apply \eqref{eq:ugly} to $\partial_\caT^r \zeta$ to find
$$
\|x^\kappa\partial_{x}^{\alpha} \partial_\caT^r \zeta (\caT)\|_{L^2(\R^3)}
\leq
c(m,s,r)
e^{\big(1+ b(m,s)
\| \widetilde\gamma\|_{_{W^{m+s,\infty}}}\big)
\caT }\ 
(1+\|\widetilde\gamma\|_{W^{r,\infty}})^r
\sum_{|\ell|=0}^s
\|x^\ell f \|_{H^{m+ s -\ell+ 2r}}  \,.
$$

\end{proof}

Given $\widetilde w_0(0)=f \in \caS(\RR^3;\KK)$ choose
$\widetilde w_0$ the solution provided by Lemma \ref{lem:schrodweight}.
It 
satisfies for all $\alpha$,
\begin{equation}
\label{eq:w0bound}
(x \, ,\partial_{\caT,t,x})^\alpha w_0\ \in
\ 
L^\infty([0,T]_\caT\times \RR_t\times\RR^3_x\, ;\, \KK)\,.
\end{equation}
Setting the right hand side of 
\eqref{eq:9.20} equal to zero yields an
equation that  is solved using a 
${\rm Hom}(\KK)$ valued integrating
factor $g(t,x)$,
\begin{equation}
\label{eq:Piw1-new}
\Pi w_1 =g(t,x)\, w_0,
\end{equation}
where $g$ is the solution of \eqref{eq:2}.
The ray average hypothesis with parameter
$0\le \beta<1$ yields
estimates for the derivatives of $g$
and therefore those of $\Pi w_1$,
\begin{equation*}
\langle t\rangle^{-\beta}\ 
(x \,,\partial_{t,x})^\alpha
(\Pi w_1)
\ \in\ {L^\infty([0,T]\times \RR_t\times\RR^3_x\,;\,\KK)}
\,.
\end{equation*} 
The component
$(I-\Pi)w_1$ is given by \eqref{eq:(I-Pi)w1} in terms of $w_0$
so
\eqref{eq:w0bound} implies,
\begin{equation*}
\label{eq:w1sublinear}
(x\,,\partial_{\caT,t,x,y})^\alpha
(I-\Pi) w_1
\ \in \ {L^\infty([0,T]\times\RR_t \times\RR^3_x\times\RR^3_y)}
\,,
\end{equation*}
with $w_1$ is periodic in $y$.
This completes the determination of
$w_0$ and $w_1$. At this stage one has $r_{-1}= r_0= \Pi \, r_1=0$.
We choose $w_2$ to that $(I-\Pi)r_1=0$. 
The latter equation holds if and only if 
\begin{equation}
\label{eq:(I-P)w2}
(I-\Pi)w_2 \ = \ \Q\,\MM w_1 + \Q\,\NN w_0 \,.
\end{equation}
This determines $(I-\Pi)w_2$. On the other hand,
$\Pi w_2$ does not affect the profiles $r_{-1},r_0,r_1$. It is 
chosen equal to zero,
\begin{equation}
\label{eq:Pw2}
\Pi\,w_2 \ =\ 0\,.
\end{equation}
The estimates for $w_0,w_1$ imply 
that the $y$-periodic
$w_2$ satisfies estimates analogous to those of
$w_1$ so,
\begin{equation}
\label{eq:wjsublinear}
\langle t\rangle^{-\beta}\ 
(x\,,\partial_{\caT,t,x,y})^\alpha
 w_j
\ \in \
L^\infty([0,T]\times\RR_t\times\RR^3_x\times\RR^3_y),
\qquad j=1,2
\,.
\end{equation}
This completes the determination of the profiles so that
\begin{equation}
\label{eq:vanishing-coeff}
r_{-1}=r_{0}=r_1=0\,.
\end{equation}

\begin{theorem}
\label{thm:diffr-approx}
If $f\in \caS(\R^3;\KK)$ there is a unique
$w_0\in C^\infty(\R;\caS(\R^3;\KK))$ satisfying 
\eqref{eq:transp-schrod} with
$w_0(0)=f$.
Define $w_1$, $w_2$ and $v^h$ by 
 \eqref{eq:(I-Pi)w1}, \eqref{eq:Piw1-new},
\eqref{eq:(I-P)w2}, \eqref{eq:Pw2} and \eqref{eq:approx-diffr} respectively.
If $u^h$ is the exact solution of $P^h u^h = 0$
with
$u^h\big|_{t=0}=v^h\big|_{t=0}$, then for all $\alpha$
\begin{equation}\label{eq:errorest-diffr}
\sup_{t\in [0,T/h]}\ 
\big\|
(x\,,\,h\,\partial_{t,x}  )^\alpha
(u^h - v^h)
\big\|_{L^2(\R^3)} \ \leq\  C(\alpha)\,h^{1-\beta}  \quad 0<h<1\,.
\end{equation}
\end{theorem}

\begin{proof}
Let $m\in\NN$. 
The bound \eqref{eq:wjsublinear} and the identity 
\eqref{eq:vanishing-coeff}
yield the 
residual estimate
\begin{equation}
\label{eq:residual-diffr}
\big\|
\langle t \rangle ^{-\beta}
(x\,,\,h\,\partial_{t,x}  )^\alpha
 \big(P^h v^h\big)
\big\|_{L^\infty([0,T/h]\times \RR^{3})}
\ \le \
C(\alpha)\, h^2,
\qquad
h\in ]0,1[\,.
\end{equation}
This combined with Theorem \ref{thm:stability1} 
and Remark \ref{rem:diffrsize} shows that there exists a constant 
$C(m,T)$ such that for all $t\in [0,T/h]$

\begin{align}
\label{eq:difference-new}
\nonumber
\big\|
(u^h-v^h)(t)\big\|_{m,h}
&\ \le\
C(m,T)\,\Big(
\big\|
(u^h-v^h)(0)\big\|_{m,h}
\ +\
\int_0^{T/h}
\big\|
P^h\,v^h(s)\big\|_{m,h}\ ds
\Big)
\\
&\ =\
C(m,T,\beta)\,\Big(
\big\|
(u^h-v^h)(0)\big\|_{m,h}
\ +\ 
\Op\big(h^{1-\beta}\big)\Big)
\,.
\end{align}
For the first term in the rhs of \eqref{eq:difference-new} we follow the proof 
of Theorem \ref{error-est1} which yields in this case
$$
\big\|(u^h-v^h)(0)\big\|_{m,h}\le \Op\big(h^{1-\beta}\big) \,.
$$
The error estimate with polynomial weights $x^{\alpha}$ requires 
an additional weighted stability estimate.
We will use the following notation. For a function $u(t,x)$, 
integers $\ell , m \geq 0$, define
$$
\|u(t)\|_{\ell,m,h} \ := \
\sum_{|\beta|\le \ell\,,\, |\alpha|\le m} \,\big\|
x^{\beta}(h\partial_{t,x})^\alpha
u(t)
\big\|_{ L^2( \RR^{3}  ) }\,.
$$

\begin{lemma}\label{lem:weighted}
Under the assumptions of Theorem \ref{thm:stability1}, there 
exist constants $c(\ell, m)$, $C(m,h)$ such that
\begin{equation}\label{eq:weight-ineq}
\|u(t)\|_{\ell,m,h}
\ \le \
c(\ell,m)\,
e^{c(\ell,m)\, C(m,h)\, t}\
\|u(0)\|_{\ell,m,h}
\,.
\end{equation}
\end{lemma}

\begin{proof}
It is a commutator argument resembling the proof of
Theorem \ref{thm:stability1}.
The proof is inductive in $\ell$. 
The case $\ell = 0$ is provided by 
Theorem \ref{thm:stability1}.
Assume that \eqref{eq:weight-ineq} holds for $\ell= \bar\ell$.
The case $\bar\ell +1$ is proved by applying the 
inductive hypothesis to $x_j u$ for $j=1,2,3$ and using the commutator 
relation 
$$
[P^h,x_j] = A_j \qquad j=1,2,3 \,.
$$
An application of Gronwall's inequality  finishes the proof of the lemma. 
\end{proof}
\noindent
Using the weighted estimates as above proves the theorem.
\end{proof}

\noindent
{\bf Proof of Theorem \ref{thm:main-diffr} } Write
$$
\uu^h-\uv^h \ =\
(u^h-v^h)
\ +\ 
( \uu^h-u^h )
\ +\
(v^h-\uv^h)\,.
$$
The preceding theorem estimates the first term.
It suffices to show that the differences $u^h-\uu^h$
and $v^h-\uv^h$ have similar upper bounds.
The first follows from the stability for Maxwell's
equations proved in Lemma \ref{lem:weighted}.
The second follows from Lemma \ref{lem:schrodweight}.
\qed

\end{document}